\documentclass[preprint,12pt]{elsarticle}

\usepackage{amssymb}
\usepackage{amsmath}
\usepackage{amsthm}
\usepackage{amsmath, amssymb, amsthm}
\usepackage{mathtools}
\usepackage{mathrsfs}
\usepackage{hyperref, a4wide, color, graphicx, comment, dsfont}
\usepackage{xstring, xparse}
\usepackage[english]{babel}
\usepackage{verbatim, autonum, booktabs}
\usepackage{faktor}
\usepackage{enumitem}   
\newtheorem{theorem}{Theorem}[section]
\newtheorem{lemma}{Lemma}[section]

\newtheorem{remark}{Remark}[section]
\newtheorem{assump}{Assumption}[section]
\DeclarePairedDelimiter\action{\langle}{\rangle} 
\newcommand{\transpose}{\top} 
\newcommand{\tild}[1]{\tilde{#1}} 
\newcommand{\R}{\mathbb{R}} 
\newcommand{\N}{\mathbb{N}} 

\newcommand{\boundary}{\partial}
\newcommand{\linspa}{\mathcal{L}}
\newcommand{\embed}{\hookrightarrow}

\newcommand{\dt}{\partial_{t}}
\newcommand{\jac}{\nabla} 
\newcommand{\lapl}{\Delta}
\newcommand{\grad}{\nabla}
\newcommand{\divv}{\nabla \cdot}
\newcommand{\curl}{\nabla \times}

\newcommand{\half}{\frac{1}{2}}
\newcommand{\Hdir}{\mathbb{H}}

\newcommand{\dom}{{D}}

\newcommand{\youE}{\tild{u}}
\newcommand{\you}{u}

\newcommand{\youn}{u^0}

\newcommand{\uD}{\tild{h}}

\newcommand{\bcop}{\mathfrak{T}}
\newcommand{\T}{\mathcal{T}}

\newcommand{\vel}{v}
\newcommand{\pre}{p}

\newcommand{\domain}{\Omega}

\newcommand{\F}{\mathcal{F}}
\newcommand{\uss}{{ u_{*} }}
\newcommand{\nonlinpar}{\lambda}

\newcommand{\pref}{p_f}
\newcommand{\velf}{v_f}
\newcommand{\pvis}{\zeta}
\newcommand{\vvis}{\mu}
\newcommand{\predir}{p_D}
\newcommand{\veldir}{v_D}

\newcommand{\preneu}{p_N}
\newcommand{\velneu}{v_N}

\newcommand{\normalvec}{n}
\newcommand{\us}{{u_*}}
\newcommand{\A}{\mathfrak{A}}
\newcommand{\B}{\mathfrak{B}}
\newcommand{\Asym}{\mathcal{A}}
\newcommand{\Aant}{\mathscr{A}}
\newcommand{\ba}{B_{1}}
\newcommand{\bb}{B_{2}}
\newcommand{\bc}{B_{3}}
\newcommand{\bd}{B_{4}}
\newcommand{\ff}{f}

\newcommand{\g}{g}
\newcommand{\energy}{\mathcal{E}}
\newcommand{\G}{\mathcal{G}}
\newcommand{\Hneu}{\mathbb{G}}
\newcommand{\confun}{{C}}
\newcommand{\Margin}[1]{}
\newcommand{\revision}[1]{\textcolor{black}{#1}}
\journal{~}

\begin{document}

\begin{frontmatter}



\title{Well-posedness of a first-order-in-time model for nonlinear acoustics with nonhomogeneous boundary conditions}


\author{Pascal Lehner} 

\affiliation{organization={Department of Mathematics},
	addressline={University of Klagenfurt}, 
	city={Klagenfurt},
	postcode={9020}, 
	state={Carinthia},
	country={Austria}}

\begin{abstract}
We study well-posedness of a first-order-in-time model for nonlinear acoustics with nonhomogeneous boundary conditions in fractional Sobolev spaces. The analysis proceeds by first establishing well-posedness of an abstract parabolic-type semilinear evolution equation. These results are then applied to concrete operators and function spaces that capture the boundary conditions relevant for realistic modeling.

Our approach is based on the spectral decomposition of a positive definite self-adjoint operator, with solution regularity characterized via the domains of its fractional powers. Employing Galerkin’s method and the Newton-Kantorovich theorem, we prove well-posedness for the abstract nonlinear system with possibly nonhomogeneous boundary data.

The connection between (spectral) fractional powers of the Laplacian and fractional Sobolev spaces due to interpolation theory allows us to transfer these results to the nonlinear acoustic model under nonhomogeneous Dirichlet and Neumann boundary conditions, yielding fractional Sobolev regularity. For Hodge/Lions boundary conditions, we establish well-posedness with solutions in classical Sobolev spaces of integer order.
\end{abstract}




\end{frontmatter}



	\section{Introduction}
In this work, we study well-posedness of the coupled nonlinear PDE
\[ \label{eq:full_system}
\begin{split}
	\dt \pre + \divv \vel - \pvis \lapl \pre + \nonlinpar \pre \divv \vel + \grad \pre \cdot \vel &= \pref \\
	\dt \vel + \grad \pre - \vvis \lapl \vel + \half \grad(\vel \cdot \vel - \pre^2) &= \velf \\
	\pre(\cdot, 0) = \pre_0, \quad \vel(\cdot,0) &= \vel_0 
\end{split}
\]
in a domain $\Omega \subset \R^d$ with $d \in \{2,3\}$ over a time interval $[0,T]$, with unknowns $p:\Omega \times [0,T] \to \R$ and $v:\Omega \times [0,T] \to \R^d$. The system is driven by source terms $p_f, v_f$ and supplemented with initial conditions $p_0,v_0$. The parameters $T, \ \nonlinpar, \ \pvis, \ \vvis >0 $ with $\nonlinpar \neq 1$ are constant. We consider equation \eqref{eq:full_system} under the following boundary conditions 
\begin{itemize} \label{bcs}
\item Dirichlet-Dirichlet 
\begin{align} \tag{2a} \label{bc-a}
\pre &= \predir \\
\vel &= \veldir \text{ on } \partial \Omega,
\end{align}
\item Neumann-Hodge/Lions ($d=3$) 
\begin{align} \tag{2b} \label{bc-b}
	\grad \pre \cdot \normalvec &= \preneu
	\\ \vel \cdot \normalvec &= \velneu \\
	( \curl \vel ) \times n &= v_\times  \text{ on } \partial \Omega,
\end{align}
\item Neumann-Dirichlet
\begin{align} \tag{2c}  \label{bc-c}
\grad \pre \cdot \normalvec &= \preneu \\ 
\vel &= v_D  \text{ on } \partial \Omega,
\end{align} 
\item Dirichlet-Hodge/Lions $(d=3)$
\begin{align} \tag{2d} \label{bc-d}
\pre &= \predir 
\\ \vel \cdot \normalvec &= \velneu
\\ ( \curl \vel ) \times n &= v_\times \text{ on } \partial \Omega,
\end{align}
\end{itemize}
where $p_D, v_D, p_N, v_N, v_\times$ are given boundary data and $n$ denotes the outer unit normal vector of $ \partial \Omega$.

In this paper, we use the notation $ x \cdot y := \sum_{i=1}^d x_i y_i$ for the Euclidean scalar product, where $x,y \in \R^d$ have components $x_i,y_i$. $\dt$ denotes the partial derivative with respect to time, acting component-wise when applied to vector valued functions. We define the spacial gradient operator as $\grad:=(\partial_{x_1}, \dots, \partial_{x_d})^\transpose$. When applied to a scalar field, $\grad$ represents the gradient vector, when applied to a vector field, it denotes the Jacobian matrix. $\grad \cdot$ denotes the divergence operator, while $\grad \times$ denotes the curl operator. The usual cross product is defined by $x \times y:= (x_2 y_3 - x_3 y_2 , x_3 y_1 - x_1 y_3 , x_1 y_2 - x_2 y_1)^\transpose$ for $x,y \in \R^3$. The Laplace operator is denoted by $\lapl:= \divv \grad$, and acts component-wise when applied to vector valued functions. 
\subsection{Motivation}
The motivation to study system \eqref{eq:full_system} under the various boundary conditions \eqref{bc-a} - \eqref{bc-d} stems from the results established in \cite{QuadraticWave}. The present work can be viewed as a continuation of the research direction initiated in \cite{QuadraticWave},  
where the authors study well-posedness of the coupled nonlinear PDE
\addtocounter{equation}{1}
\[ \label{eq:full_system_0}
\begin{split}
	\dt \pre + (1 + \gamma)\divv \vel - \pvis \lapl \pre + \epsilon_1 \pre \divv \vel + \epsilon_2 \grad \pre \cdot \vel &= \pref \\
	\dt \vel + (1 + \delta )\grad \pre - \vvis \lapl \vel - \frac{\epsilon_3}{2} \grad(\pre^2) + \frac{\epsilon_4}{2} \grad(\vel \cdot \vel) &= \velf \\
	\pre(\cdot, 0) = \pre_0, \quad \vel(\cdot,0) &= \vel_0 \\
	\pre = 0, \quad \vel &= 0 \text{ on } \partial \Omega 
\end{split}
\]
with $\epsilon_1, ..., \epsilon_4 > 0$, and $\gamma, \delta \in L^\infty(\Omega)$ sufficiently small given,
see \cite[Eq. (13)]{QuadraticWave}.
Equation \eqref{eq:full_system_0} can be interpreted as a reduced model derived from the Navier–Stokes–Fourier system. It serves as an approximation for nonlinear acoustic wave propagation in a viscous fluid, capturing the interaction between local fluctuations in pressure $p$ and the velocity field $v$, while not considering thermal and density variations. For a detailed derivation of \eqref{eq:full_system_0} from fundamental physical principles, including the conservation of mass, momentum, and energy, as well as thermodynamic laws, we refer the reader to \cite[Section 2]{QuadraticWave}.

To place equation \eqref{eq:full_system_0} in the context of wave equations and highlight its relevance to nonlinear acoustics, we note that it can be interpreted as a nonlinear, damped extension of the hyperbolic linear wave equation
\[ \label{eq:linsys}
\begin{split}
	\dt p + \divv v = p_f \\
	\dt v + \grad p = v_f .
\end{split}
\] 
Equation \eqref{eq:full_system_0} extends this linear system by incorporating nonlinear terms and viscous dissipation, thereby enabling the modeling of high-intensity acoustic phenomena in viscous fluids.
In fact, the damping parameters 
 $\zeta, \mu$ are typically small in the context of nonlinear acoustics, so that the wave like behavior induced by the operator pair 
 $(\divv, \grad)^\transpose$
dominates over the diffusive effects introduced by the Laplacian $- \lapl$.

Moreover, as observed in \cite[Remark 2.4]{QuadraticWave}, setting $\gamma = \delta = 0$ in equation \eqref{eq:full_system_0} yields a model that can be viewed as an alternative first-order-in-time formulation of Kuznetsov's equation. The latter is given by
\revision{
\[ 
\label{Kuz}
\begin{split}
\partial_{tt} p - c^2 \lapl p - b \dt \lapl p - \varepsilon \frac{B}{  2 A} \partial_{tt}  ( p^2 ) - \varepsilon \partial_{tt} (v \cdot v) &= 0
\\
\dt v + ´\grad p &= 0 \\
p(\cdot, 0) = p_0, \quad \dt p(\cdot, 0) = p_1, \quad v(\cdot, 0) &= v_0 \\
\pre &= 0 \text{ on } \partial \Omega 
\end{split} 
\]
}
\Margin{4}
and originally derived in \cite{kuznetsov:1971}.
Here $c$ denotes the wave speed, $b$ is a damping parameter, $\varepsilon$ a small parameter and $\frac{B}{ A}$ is the parameter of nonlinearity.
Kuznetsov's equation is widely used in nonlinear acoustics as a canonical model for high-intensity sound propagation. For a physical introduction to nonlinear acoustics, we refer to \cite{Hamilton:1997}, while a comprehensive mathematical overview can be found in the review article \cite{Kaltenbacher:2015}.

In contrast to the model \eqref{eq:full_system} considered in the present work, Kuznetsov’s equation involves second-order time derivatives, including within the nonlinear terms. \revision{Consequently, the regularity theory differs from that developed here, see, \cite{Kaltenbacher:2015, MeyerWilke2013}. In particular, \cite[Theorem 1]{MeyerWilke2013} establishes optimal $L^p$-regularity results for \eqref{Kuz}. In the Hilbert space setting $p=2$, assuming $p_0 \in H^2(\Omega)$, $p_1, v_0 \in H^1(\Omega)$, the authors obtain $p \in H^2(0,T;L^2(\Omega)) \cap H^1(0,T;H^2(\Omega)) $, $v \in BUC^1(0,T;H^1(\Omega))  $ and $\dt v \in H^{\frac{3}{2}}(0,T;L^2(\Omega)) \cap H^1(0,T;H^1(\Omega))$. Here, $BUC^1$ denotes the space of functions having a bounded and uniformly continuous derivative.
}
\Margin{4}

Interest in first-order formulations stems largely from their practical advantages in the design of efficient numerical methods, particularly in the context of adaptive space-time discretizations. In the case of the linear acoustics equation \eqref{eq:linsys}, effectiveness has already been demonstrated as a foundation for discontinuous Galerkin methods, \revision{ see \cite{DoerflerFindeisenWienersZiegler:2019, bansal:2021}}.
\Margin{1.}

To employ equation \eqref{eq:full_system_0} as a physically meaningful model for nonlinear wave propagation, it is essential to impose boundary conditions that reflect realistic scenarios. The goal of this paper is therefore to analyze the system \eqref{eq:full_system_0} under a variety of more general and physically relevant boundary conditions, rather than restricting the analysis to the commonly used homogeneous Dirichlet conditions. The specific choices of boundary conditions considered in this work ranging from \eqref{bc-a} to \eqref{bc-d} are explained below in section \ref{se:bcs}.

Before proceeding, we comment on the choice of the parameters used to derive \eqref{eq:full_system} from the more general system \eqref{eq:full_system_0}. First, we set $\epsilon_1=\lambda, \epsilon_2 = 1, \epsilon_3 = 1, \epsilon_3 = 1, \epsilon_4 = 1$, a choice justified by the transform of variables $\tild{p} = \varepsilon p$ and $\tild{v}= \varepsilon v$, cf. \cite[Eq. (12)]{QuadraticWave}. In \cite{QuadraticWave} the parameters $\epsilon_1, \dots, \epsilon_4$ are treated as placeholders, allowing for flexibility in the coefficients of the nonlinear terms. However, in this work, we choose to simplify the model by avoiding additional parameters and instead focus on the core structure of the equation relevant to nonlinear acoustics.
Second, we set $\gamma=\delta=0$, since, as noted earlier, Kuznetsov's equation can be recovered from equation \eqref{eq:full_system} under this assumption. Moreover, the terms involving 
$\gamma,\delta$ are related to entropy contributions, which are typically neglected in models of nonlinear acoustics. While the analysis presented here could, in principle, be extended to the case 
$\gamma,\delta \neq 0$ following similar arguments as in \cite{QuadraticWave}, we omit this extension to streamline the exposition and reduce notation.


\subsection{Contributions}
First, this paper extends the results of \cite{QuadraticWave} by proving existence of $H^{1+s}(\Omega)$ regular solutions in space for $s \in [\half, 1]$ in the case of nonhomogeneous Dirichlet boundary conditions, see Theorem \ref{th:wellposeddir}.
Compared to \cite{QuadraticWave}, the approach presented here accommodates slightly less regular initial data $(p_0,v_0)$ and right-hand sides $(\pref, \velf)$. An immediate consequence is that, for the limiting case $s=\half$ solutions that are unbounded are not necessarily excluded from a theoretic standpoint, since $H^s(\Omega) \embed L^\infty(\Omega)$ only for $s>\frac{3}{2}$ if $d=3$. To the author’s knowledge, fractional regularity results of this type have not been investigated for other equations in nonlinear acoustics.

Beyond the Dirichlet setting, this work also establishes well-posedness of equation~\eqref{eq:full_system} under nonhomogeneous Neumann and Hodge/Lions boundary conditions, see Theorems \ref{th:wellposedneu}, \ref{th:wellposednoslip}, and \ref{th:dir-hod}. The analysis of Hodge/Lions boundary conditions requires more delicate treatment, and in this case, we restrict to spatial dimension 
$d=3$ and prove well-posedness for solutions with 
$H^2(\Omega)$ regularity in space. To the best knowledge of the author, this is the first work to address Hodge/Lions boundary conditions in the context of nonlinear acoustics. \\

To obtain these results, we consider the more abstract formulations \eqref{eq:full_system_abstract} and \eqref{eq:full_system_abstract_nonlinear}, which are semilinear parabolic evolution equations that have been extensively studied in the literature, see, for example, \cite{Ladyzhenskaya1968, lunardi1995}.

For instance, well-posedness of \eqref{eq:full_system} can be established using standard semigroup techniques, as presented in \cite{lunardi1995, pazy1983}.	
By applying classical semigroup theory, well-posedness of equation \eqref{eq:full_system_abstract_nonlinear} follows under the condition $g \in D(\Asym^\gamma)$ with $\gamma>\frac{3}{4}$, see \cite[Section 8.3, Theorem 3.5]{pazy1983}. While there also exist results for less regular initial data, these typically require stronger assumptions on the right-hand side $f$, cf. \cite[Theorem 7.1.6]{lunardi1995}, while in this work the right hand side can be any function in $L^2$, thereby including discontinuities in both space and time.

Alternatively, well-posedness of equation \eqref{eq:full_system} can be shown via energy methods, as in \cite{Ladyzhenskaya1968}. However, the analysis in \cite{Ladyzhenskaya1968} only considers fractional Hölder regularity and does not address fractional Sobolev regularity. While fractional Sobolev regularity results are provided in \cite{Ribaud1998}, those results are limited to the whole space $\R^d$.
To the best knowledge of the author, a higher-order fractional Sobolev energy estimate on a bounded domain, combined with the specific type of nonlinearity considered here, has not yet been addressed in the literature.


\subsection{Choice of boundary conditions} \label{se:bcs}
In the following, we briefly motivate our choice of Hodge/Lions boundary conditions for $v$ instead of classical Neumann conditions on each component of 
$v$, as one might naturally impose for parabolic-type equations.

By Hodge/Lions boundary conditions we refer to the two boundary conditions involving the velocity $
v$. In the Navier–Stokes literature, these are also known as Hodge, Lions, or Navier-slip boundary conditions, see, e.g., \cite{li2022b, monniaux2013, phan2017, xiao2007}.
Physically, such conditions are often used to model friction along the boundary of the domain. The specific Hodge/Lions formulation considered here is a special case of slip boundary conditions. This choice is particularly natural for equation~\eqref{eq:full_system}, which is derived from the Navier–Stokes–Fourier system for viscous, heat-conducting fluids, cf. \cite[Section~2]{QuadraticWave}.

If slip effects are not to be modeled, one may instead impose Dirichlet conditions on each component of $v$, leading to the classical no-slip condition $v=0$ on $\partial \Omega$. See \cite{li2022b} for an overview of different boundary conditions used for the Navier–Stokes equations.
The present work therefore treats both slip and no-slip boundary conditions, while also allowing a choice between Dirichlet and Neumann conditions for 
$p$.

On the other hand, one could also consider classical Neumann boundary conditions on each component of $v$. However, the physical meaning of such conditions is questionable. Moreover, due to the fact that skew symmetry of $(\divv, \grad)^\transpose$ is lost, classical Neumann boundary conditions on $v$ do not fit into the framework that is used here. 
Therefore, the  Hodge/Lions boundary conditions considered here are more natural for equation \eqref{eq:full_system} as they provide differential operators and energy estimates in line with the Dirichet case. \\


The rest of the paper is organized as follows. In order to treat the various boundary conditions in a unified way, we first consider a more abstract PDE in section \ref{se:abstractsystem} and prove its well-posedness in Theorems \ref{th:nonlinear} and \ref{th:wellposednonhom}. Then, we apply the results to equation \eqref{eq:full_system} under the various boundary conditions \eqref{bc-a} to \eqref{bc-d} in section \ref{se:application}.
\section{Abstract system} \label{se:abstractsystem}
We start by showing well-posedness of the auxiliary linear Cauchy problem  
\[ \label{eq:full_system_abstract}
\begin{split} 
	\dt u + \A u + \B[u ,\us ] + \B[\us ,u ] &=  \ff \\
	u(0) &= \g
\end{split}
\] 
with $\A$ linear, $\B$ bilinear, and given data $\us, \ff, g$ in section \ref{se:GalerkinApprox}. Then, in sections \ref{se:nonlinearproblem} and \ref{nonhombcs} using well-posedness of \eqref{eq:full_system_abstract}, we provide well-posedness of the nonlinear problem
\[ \label{eq:full_system_abstract_nonlinear}
\begin{split} 
	\dt u + \A u + \B[u , u ] &=  \ff \\
	u(0) &= \g 
\end{split}
\] 
with possibly nonhomogeneous boundary conditions.

In the following, we denote the ball of radius $r$ centered at $y \in X$ by $B_{r}^{X}(y) := \{ x \in X : \| x -y \| \leq r \}$ for a normed space $X$. With $\linspa(X,Y)$ we denote the space of linear and bounded functions from Banach spaces $X$ to $Y$ and denote a continuous embedding of such spaces as $X \embed Y$. With $X^*$ we denote the dual space of $X$ and $L^2(0,T;X)$ or $H^1(0,T;Y)$ denote the usual Bochner spaces. Continous functions from $[0,T]$ to $X$ are denoted as $C([0,T];X)$. Furthermore, we refer to
\[
a b \leq \frac{\delta a^2 }{2} + \frac{b^2}{2 \delta} \quad a,b,\delta>0
\]
as Young's inequality. 

In the analysis, we rely on the following assumptions, which are supplemented by Remark \ref{re:fractional}, see below.
\begin{assump} \label{se:assumption}
	~
\begin{enumerate}[label=(\roman*),ref=\theassump(\roman*)]
	\item\label{se:assa} $H$ is a real infinite dimensional separable Hilbert space with scalar product $\action{\cdot, \cdot}$. 
	\item\label{se:assb} The linear densely defined operator $\A:\dom(\A) \subset H \to H$ is the sum of a linear self adjoint operator $\Asym$ and a linear skew symmetric operator $\Aant$. 
	
	More precisely, $\Asym:\dom(\A) \subset H \to H$ is self adjoint and has a spectral decomposition with the following properties. $\Asym$ has pure point spectrum, and there exists a sequence of real eigenvalues $\{ \lambda_k \}_{k=1}^\infty$ such that
	$0 < \lambda_1 \leq ... \leq \lambda_k  \leq ... $ and $\lambda_k \to \infty$ for $k \to \infty$ holds. Furthermore, the set of eigenvectors $\{ \phi_k \}_{k=1}^\infty$
	forms an orthonormal Hilbert basis of $H$. 
	
	For $\Aant:\dom(\A) \subset H \to H$, we assume skew symmetry
	which in particular leads to
	\[ 
	\label{as:anti}\action{\Aant u, u} = 0 \quad \text{ for }  u \in \dom(\A) .
	\] 
	Moreover, $\Aant: \dom(\Asym^\half) \to H$ is required to be bounded with operator norm $C_{\Aant}>0$.
	\item\label{se:assc} $ \B : \dom(\A) \times \dom(\A) \to H $ is bilinear and satisfies
	\[ \label{as:bilinestimate1} \vert \action{ \B[u,w], v } \vert \leq  C_\B \|  \Asym^\frac{\sigma + 1}{2} u \|_H \| \Asym^{\frac{\sigma }{2}} w \|_H \| \Asym^{\frac{1 - \sigma}{2}} v \|_H 
	\]  
	for some $\sigma \in [0,1]$ and any $u \in \dom(\Asym^\frac{\sigma+1}{2}) \ , w \in \dom(\Asym^\frac{\sigma}{2}), \ v \in \dom(\Asym^\frac{1-\sigma}{2})$ with $C_\B>0$.
\end{enumerate}
\end{assump}
\begin{remark} \label{re:fractional}
	Let Assumption \ref{se:assumption} hold. Due to the positivity of the eigenvalues of $\Asym$, we define fractional powers $s \in \R$ of $\Asym$ via the spectral decomposition as
	\[
	\Asym^s u : \dom(\Asym^s) \subset H \to H  , \quad \Asym^s u := \sum_{k=1}^\infty \lambda_k^s \action{u, \phi_k} \phi_k
	\]
	and the operators $\{ \Asym^s \}_{s \geq 0}$ are themselves again self adjoint 
	with domain 
	\[  \dom(\Asym^s) =  \left\{ u 
	\in H   : \sum_{k=1}^\infty  \lambda_k^{ 2s} \action{u, \phi_k}^2  < \infty \right\}, \]
	see \cite[Section 5.8]{zeidler1995} for details. 
	
	Defining $\dom(\Asym^{-s}) := \dom(\Asym^s)^*$ the domains are 
	Hilbert spaces for every $s \in \R$ with norm 
	$$\| u \|_{\dom(\Asym^s)} := \left(\sum_{k=1}^\infty \lambda_k^{2s} \action{u, \phi_k}^2 \right)^\half = \| \Asym^s u  \|_H 
	$$ and therefore $\Asym^s: \dom(\Asym^t) \to \dom(\Asym^{t-s})$ continuously for any $t \in \R$ cf. \cite[Section 2.1]{fefferman2022}.
	
	Since $\lambda_k \to \infty$, there exists $N \in \N$ such that $\lambda_n \geq 1$ for all $n \geq N$. From this we can conclude that for $s \geq t \geq 0 $
	\[ \label{as:frac_embed} \dom(\Asym^s) \embed \dom(\Asym^t)\]
	with embedding constant denoted by $c_\Asym(s,t) > 0$.
	
	Moreover, there is a characterization of the fractional domains as interpolation spaces, namely 
	\[ \label{eq:interpolation}
	\dom(\Asym^\theta) = (H, \dom(\Asym) )_\theta
	\] with $0 < \theta < 1 $ denoting the parameter of real interpolation on Hilbert spaces, see \cite[Theorem 4.36]{Lunardi2018}.
	
	Additionally, we have that the embedding \[ \label{as:Linfemb} 
	L^2(0,T;\dom(\Asym^\frac{\sigma+1}{2})) \cap H^1(0,T;\dom(\Asym^{\frac{1 - \sigma}{2}})^*) \embed C([0,T];\dom(\Asym^\frac{\sigma}{2})) 
	\]
	is continuous for any $T>0$ 
	with embedding constant $C_T>0$ depending on $T$. This follows, since $\dom(\Asym^{\frac{\sigma}{2}})$ is a Hilbert space and from \eqref{as:frac_embed} we have $\dom(\Asym^{\frac{\sigma}{2}}) \embed \dom(\Asym^{\frac{\sigma + 1}{2}})$, yielding the Gelfand tripel $  \dom(\Asym^{\frac{\sigma}{2}}) \subset \dom(\Asym^{\frac{\sigma + 1}{2}}) \subset \dom(\Asym^{\frac{\sigma + 1}{2}})^*$. Applying \cite[Lemma 7.3, $p=p'=2$]{Roubicek:2013} yields
	$  	L^2(0,T;\dom(\Asym^\frac{\sigma+1}{2})) \cap H^1(0,T;\dom(\Asym^{\frac{1 + \sigma}{2}})^*) \embed C([0,T];\dom(\Asym^\frac{\sigma}{2}))  $. The embedding \eqref{as:Linfemb} then follows since, $ \dom(\Asym^{\frac{\sigma - 1}{2}})^* = \dom(\Asym^{\frac{- \sigma + 1}{2}}) \embed \dom(\Asym^{\frac{- \sigma - 1}{2}})  =\dom(\Asym^{\frac{\sigma + 1}{2}})^*$.
\end{remark}
\subsection{Auxiliary linear problem} \label{se:GalerkinApprox} 
The goal of this section is to prove the following theorem under the assumptions introduced above, which provides well-posedness of the linear system \eqref{eq:full_system_abstract}.
Note that the assumed structure of $\A$ is crucial in the following results that are based on energy estimates. While the self-adjointness allows us to obtain a positive energy, the skew-adjoint terms cancel out and are not part of the energy.

\begin{theorem} \label{th:linearwellposed}
	Let $T>0$ and Assumption \ref{se:assumption} hold with $\sigma \in [0,1]$.
	Then, for small enough $r_G>0$ and all $\us \in B_{r_G}^{X}(0)$ as well as any
		$$f \in Y := L^2(0,T; D(\Asym^{\frac{1-\sigma}{2}} )^*), \quad g \in \dom(\Asym^\frac{\sigma}{2}) $$
	there exists a unique solution 
	$$u \in	X := L^2(0,T;D(\Asym^{\frac{1+\sigma}{2}}) )  \cap H^1(0,T; D(\Asym^{\frac{1-\sigma}{2}} )^*)  $$ 
	of equation \eqref{eq:full_system_abstract}. Moreover, the a priori estimate
	\[ \label{eq:GalerkinEstimate_Dir}
	\| u \|_X \leq C_G (  \| \ff \|_{Y} + \| g \|_{\dom(\Asym^\frac{\sigma}{2}) } )
	\] holds for some $C_G>0$.
\end{theorem}
\begin{proof}
	The proof is split into several steps. To shorten notation we define $U:= D(\Asym^{\frac{1+\sigma}{2}})$, $V:=   D(\Asym^{\frac{1-\sigma}{2}} )$ and $W:=  \dom(\Asym^\frac{\sigma}{2}) $.
	
	First, we show that $\A:X \to Y$ is bounded. 
	Due to assumption \ref{se:assb} and embedding \eqref{as:frac_embed}, we know that $\Asym: U \to V^*$ is bounded with some constant $C_{\Asym}>0$, since $U$ and $V^*$ can be expressed in terms of fractional domains of $\Asym$. Furthermore, by boundedness of $\Aant:\dom(\Asym^\half) \to H$ and \eqref{as:frac_embed} from Remark \ref{re:fractional} 
	$$ \vert \action{ \Aant u, v } \vert \leq \| \Aant u \|_H \| v \|_H \leq C_{\Aant} \| u \|_{\dom(\Asym^\half)} \| v \|_{H} \leq C_{\Aant} c_{\Asym}(\tfrac{1+\sigma}{2},\tfrac{1}{2}) c_{\Asym}(\tfrac{1-\sigma}{2},0) \| u \|_{U} \| v \|_{V}
	$$
	yielding with $c_\Asym^2:=  c_{\Asym}(\frac{1+\sigma}{2},\half) c_{\Asym}(\frac{1-\sigma}{2},0)$ 
	$$\int_0^T \| \A u \|_{V^*}^2 ds \leq \int_0^T ( \| \Asym u\|_{V^*} + \| \Aant u\|_{V^*} )^2 ds \leq ( C_{\Asym} + C_{\Aant}c_{\Asym}^2)^2\int_0^T  \| u \|_U^2 ds,
	$$
	which when taking square roots on both sides, gives boundedness of $\A:X \to Y$ with constant $C_\A \leq (C_{\Asym} + C_{\Aant} c_{\Asym}^2)$. 
	
Here, $ds$ denotes integration with respect to the time variable. 
Throughout this paper, the operator $\A$ is understood to act pointwise in time. That is, for a time dependent function $u$, we interpret $\A u$ as the mapping $t \mapsto \A u(t)$. 
To simplify notation, we write $c_\Asym$ for any constant of the form $c_\Asym(x,y)$ without explicitly indicating the dependence on $x, y \in \R$.

	Second, we have boundedness of $\B : X \times X \to Y$ with constant less or equal then $C_\B C_T$. This is due to assumption \ref{se:assc} implying \eqref{as:bilinestimate1} and \eqref{as:Linfemb}, since then
	\[
	\int_0^T \| \B[u,v] \|^2_{V^*} ds \leq \int_0^T C_\B^2 \| u \|_U^2 \| v \|^2_W ds \leq \max_{s \in [0,T ]} \| v \|_W^2 C_\B^2 \int_0^T \| u \|_U^2 ds \leq C_T^2 C_\B^2 \| u \|_X^2 \| v \|_X^2.
	\]
	Again, $\B$ is understood as acting pointwise in time, that is, by $\B[u,v]$ we mean $t \mapsto \B[u(t),v(t)]$ if $u,v$ are $t$-dependent. 
	\subsubsection{Galerkin Approximation}
	To show existence of a solution, we use Galerkin's method. That is, we \revision{approximate solutions to \eqref{eq:full_system_abstract} by 
	\( u_m(t) = \sum_{k=1}^m \xi_k^m(t) \phi_k \in U_m=\text{span}\{ \phi_k : k=1, \dots, m\} \subset U \) satisfying 
	\[ \label{eq:galerkin_equation}
	\action{	\dt u_m + \A u_m + \B[u_m ,\us ] + \B[\us ,u_m ] , \phi} = \action{ \ff, \phi} , \quad u_m(0) = g_m
	\]
	for all $\phi \in U_m \subset V$. Here $g_m$ is a projection of $g$ onto $U_m$.
	} 
	\Margin{2}
	The construction of the sequence $\{u_m \}_{m \in \N}$ is well posed in the usual sense with absolutely continuous coefficients $\xi_k^m$. 
	We omit the details here and refer to e.g. \cite[Chapter 7]{Evans:2010} and \cite[Section 2.3.1]{QuadraticWave}. 
	\subsubsection{Energy estimates}
	For the energy estimates, we define the time dependent energy functional 
	$$\energy[{u_m}](t) :=  \max_{s \in [0,t]} \| u_m(s) \|^2_H + \int_0^t \| \Asym^\half u_m(s)\|^2_H ds  
	$$
	and, as a first step, prove its boundedness uniformly in $m \in \N$. The initial energy $\energy[\cdot](t=0)$ is denoted by $\energy_0[\cdot]$. 
	
	Testing equation \eqref{eq:galerkin_equation} with $\phi(s) := c_{\Asym}^2 C_{\Aant}^{2} u_m(s) + \Asym^\sigma u_m(s) \revision{ \, \, \in U_m}$\Margin{2}, where $\sigma, c_{\Asym}, C_{\Aant}$ as in Assumption \ref{se:assumption}, gives
	\[ \label{eq:energyestimate0}
	\action{	\dt u_m(s) + \A u_m(s) + \B[u_m(s) ,\us(s) ] + \B[\us(s) ,u_m(s) ] , \phi(s)} = \action{ \ff(s), \phi(s)}.
	\]
	Using assumption \ref{se:assb} that $\Asym$ and its fractional powers are self adjoint and $\Aant$ is skew symmetric \eqref{as:anti}, while keeping in mind that $u_m(s) \in \dom(\Asym)$ is absolutely continuous with respect to time, equation \eqref{eq:energyestimate0} implies
	\[ \label{eq:energestimate1}
	\begin{split}
		c_{\Asym}^2 C_{\Aant}^{2} &\left( \frac{\dt}{2} \| u_m \|_H^2 + \| \Asym^{\frac{1+0}{2}} u_m \|_H^2\right) + \left( \frac{\dt}{2}  \| \Asym^{\frac{0+\sigma}{2}}u_m \|_H^2 +  \| \Asym^\frac{1+\sigma}{2} u_m \|_H^2 \right) \\
		&\qquad \qquad \leq  \vert \action{\Aant u_m, \Asym^\sigma u_m} + \action{ \B[u_m ,\us ] + \B[\us ,u_m ] + \ff, \phi} \vert,
	\end{split}
	\]
	where we omit the time argument $s$ of $u_m(s)$ in the following for reasons of readability.
	Using boundedness of $\Aant$, Young's inequality 
		and \eqref{as:frac_embed} with $\sigma \leq \frac{\sigma+1}{2}$ we get 
		\[  \vert \action{\Aant u_m, \Asym^\sigma u_m} \vert
		\leq  \| \Aant u_m \|_H \|  \Asym^\sigma u_m \|_H \leq
		\frac{c_{\Asym}^2 C^{2}_{\Aant}}{2} \| \Asym^\half u_m \|_H^2 +  \frac{1}{2} \| \Asym^{\frac{1+\sigma}{2}}u_m\|^2_H , \]
		so that after multiplying $\eqref{eq:energestimate1}$ by $2$ we are left with
		\[ \label{eq:energyestimate2}
		\begin{split}
			C_A ( \dt \| u_m \|_H^2 +  \| \Asym^\half u_m \|_H^2) &+  \dt  \| \Asym^{\frac{\sigma}{2}} u_m \|_H^2 + \| \Asym^\frac{1+\sigma}{2} u_m \|_H^2 
			\\
			&\leq 2 \vert \action{ \B[u_m ,\us ] + \B[\us ,u_m ] + \ff, \phi} \vert
		\end{split}
		\]
		with $C_A=c_{\Asym}^2 C_{\Aant}^2$.
		Upon integrating from $0$ to $t$ with respect to $ds$, and then taking the maximum of $s$ from $0$ to $t$ in \eqref{eq:energyestimate2}, we obtain 
		\[ \label{eq:energyestimate3}
		\begin{split}
			C_A ( \energy[u_m](t) - \energy_0[u_m] )   + &  \energy[\Asym^\frac{\sigma}{2} u_m](t)-  \energy_0[\Asym^\frac{\sigma}{2} u_m]
			\\
			&\leq 2 \int_0^t \vert \action{ \B[u_m ,\us ] + \B[\us ,u_m ] + \ff, \phi} \vert ds. 
		\end{split}
		\]
		Using \eqref{as:frac_embed} with $\frac{1-\sigma}{2} \leq \frac{1+ \sigma}{2}$ and Young's inequality twice we get
		\[  
		\begin{split}
			2 \int_0^t \vert \action{ \ff, \phi} \vert ds &\leq 2 \int_0^t \| \ff \|_{V^*} \| \Asym^{\frac{1-\sigma}{2}} \phi \|_H ds
			\\
			&\leq 2 \int_0^t\| \ff \|_{V^*} ( c_{\Asym} C_A +1) \| \Asym^{\frac{1-\sigma}{2} + \sigma} u_m \|_H) ds
			\\
			&\leq \frac{1}{ \delta} \int_0^t  \| f\|^2_{V^*} ds +  {\delta} \int_0^t ( c_{\Asym} C_A + 1 )^2 \| \Asym^{\frac{1+\sigma}{2}} u_m \|_H^2 ds 
			\\
			&\leq   \frac{ 1}{ \delta}  \| f \|_Y^2 + \frac{1}{2}  \energy[\Asym^\frac{\sigma}{2} u_m](t)
		\end{split}
		\] 
		with $\delta = \half ( c_{\Asym} C_A + 1)^{-2} $ resulting, when used in \eqref{eq:energyestimate3}, in 
		\[ \label{eq:energyestimate4}
		\begin{split}
			{C_A} \energy[u_m](t) + \frac{1 }{2} \energy[\Asym^\frac{\sigma}{2} u_m](t) &-  C_A \energy_0[u_m]  -  \energy_0[\Asym^\frac{\sigma}{2} u_m]  \\
			&\leq  2  \int_0^t \vert \action{ \B[u_m ,\us ] + \B[\us ,u_m ], \phi} \vert ds +  {\delta}^{-1} \| f \|_Y^2.
		\end{split}
		\]
		Using the estimate \eqref{as:bilinestimate1} on $\B$, Young's inequality, the embedding \eqref{as:Linfemb} $X \embed C([0,T];W)$ and \eqref{as:frac_embed} with $\frac{1-\sigma}{2} \leq \frac{\sigma+1}{2}$ we get
		\[ \label{eq:energyestimateB}
		\begin{split}
			2 &\int_0^t \vert \action{ \B[u_m ,\us ] + \B[\us ,u_m ], \phi} \vert ds 
			\\
			&\leq 2 \int_0^t  C_\B  (
			\| \Asym^{\frac{1+\sigma}{2}} u_m \|_H \| \Asym^{\frac{\sigma}{2}} \us \|_H  + 
			\| \Asym^{\frac{1+\sigma}{2}} \us \|_H \| \Asym^{\frac{\sigma}{2}} u_m \|_H   ) \| \Asym^{\frac{1-\sigma}{2}} \phi \|_H ds \\
			&\leq  \int_0^t \frac{C_\B^2}{ \|\us \|_X} (
			\| \Asym^{\frac{1+\sigma}{2}} u_m \|_H^2 \| \Asym^{\frac{\sigma}{2}} \us \|_H^2  + 
			\| \Asym^{\frac{1+\sigma}{2}} \us \|_H^2 \| \Asym^{\frac{\sigma}{2}} u_m \|_H^2 ) + {\|\us \|_X} \| \Asym^{\frac{1-\sigma}{2}} \phi \|_H^2 ds \\
			&\leq   {C_\B^2} \| \us \|_X ( C_T^2 \int_0^t  
			\| \Asym^{\frac{1+\sigma}{2}} u_m \|_H^2 ds   + \max_{s \in [0,t]} \| \Asym^{\frac{\sigma}{2}} u_m \|_H^2 )
			\\ &\qquad \qquad \qquad \qquad \qquad \qquad \qquad \quad + (c_{\Asym} C_A + 1)^2 {\|\us \|_X} \int_0^t  \| \Asym^{\frac{1+\sigma}{2}} u_m \|_H^2 ds 
			\\
			&\leq   \| \us \|_X   ( \max\{C_T^2, 1\} C_\B^2 +  (c_{\Asym} C_A + 1)^2 ) \energy[\Asym^{\frac{\sigma}{2}} u_m](t)   
		\end{split}
		\]
		and by choosing $r_G>0$ such that $\| \us \|_X \leq r_G \leq \left( {4(\max\{C_T^2, 1\} C_\B^2 +  (c_{\Asym} C_A + 1)^2)} \right)^{-1} $ we deduce from \eqref{eq:energyestimate4} that 
		\[ \label{eq:energyestimate5}
		{C_A} \energy[u_m](t) + \frac{1 }{4} \energy[\Asym^\frac{\sigma}{2} u_m](t) 
		\leq   \delta^{-1} \| f \|_Y^2 + ( 1 +C_A c_{\Asym}^2) \| u_m(0) \|^2_{W}.
		\]
		In particular, due to \eqref{eq:energyestimate5}, $\| u_m \|^2_{L^2(0,T;U)} \leq C_U ( \| f \|_Y^2 + \| u_m(0) \|^2_{W} ) $ with $C_U= 4 \max\{\delta^{-1}, (1+ C_A c_{\Asym}^2) \}$. \revision{To estimate $ \| \dt u_m\|_{V^*}$,
		we proceed similarly to \cite[Chapter 7.1]{Evans:2010}. Let $\phi \in V \subset H$ with $\| \phi \|_V \leq 1$. We write $\phi=\phi^1 + \phi^2$ with $\phi^1 \in U_m $ and $\phi^2 \in U_m^\perp$, where the orthogonal complement is taken with respect to the inner product of $H$. Since $\partial_t u_m \in U_m \subset H$, we infer from equation \eqref{eq:galerkin_equation} 
		\[
		\begin{split}
	\action{	\dt u_m , \phi}_V = \action{	\dt u_m , \phi^1} &= \action{ \ff -  \A u_m - \B[u_m ,\us ] - \B[\us ,u_m ], \phi^1} \\
	 &\leq \| \ff -  \A u_m - \B[u_m ,\us ] - \B[\us ,u_m ] \|_{V^*} \| \phi^1 \|_V.
\end{split}
		\]
		Moreover, we have $$\| \phi^1 \|_V^2 = \| \Asym^{\frac{1- \sigma}{2}} \phi^1\|_H^2 = \| \sum_{k=1}^m \lambda_k^{\frac{1- \sigma}{2}} \action{ \phi, \phi_k} \phi_k \|^2_H = \sum^m_{k=1} \lambda_k^{1 - \sigma} \action{\phi , \phi_k}^2 \leq \| \phi \|_V^2, $$ where the last equality follows from Parseval’s theorem. Thus, using \eqref{as:bilinestimate1}, \eqref{as:Linfemb}
				}
		\Margin{2}
		\[ 
		\begin{split}
			\| \dt u_m \|_{V^*} &\leq \| \A u_m \|_{V^*} + \| \B[u, \us] \|_{V^*} + \| \B[\us, u] \|_{V^*} + \| f \|_{V^*} \\
			&\leq C_\A \| u_m \|_U + C_\B  \| \us \|_X  ( \| u_m \|_U +\| u_m \|_W ) + \| f \|_{V^*} \\
			&\leq C_V \| u_m \|_U + \| f \|_{V^*}
		\end{split} 
		\] 
		with $C_V=C_\A  +  C_\B (1 +c_{\Asym})  r_G $. The above estimate, when squared and integrated from $0$ to $T$, leads to 
		\[ 
		\| \dt u_m \|_{L^2(0,T;V^*)}^2 \leq 2 C_U C_V^2 (\| f \|_Y^2 + \| u_m(0) \|^2_{W}) + \| f \|_Y^2 = 2 (C_U C_V^2 + 1) ( \| f \|_Y^2 + \| u_m(0) \|^2_{W} ) . 
		\] 
		Thus, combing this with \eqref{eq:energyestimate5} we obtain the a priori estimate \[ \| u_m \|_X \leq C_G ( \| f\|_Y + \| u_m(0) \|_W) \] for $C_G = \sqrt{\max\{ C_U, 2(C_U C_V^2 + 1)\}} > 0$. 
		\subsubsection{Existence and uniqueness}
		Since $X$ is Hilbert and $u_m$ bounded in $X$, there exists a weakly convergent subsequence of $u_m$ converging weakly to some $u \in X$. 
		
		Let $u_m$ denote this subsequence. Then, we have for all $v \in  L^2(0,T;V) \subset X$
		\[ \label{eq:weakconvergence}
		\begin{split}
		\action{f , v}_Y &= \lim_{n \to \infty} \action{\dt u_m + \A u_m + \B[u_m,\us] + \B[\us, u_m] , v}_Y 
		\\
		&=  \action{\dt u + \A u + \B[u,\us] + \B[\us, u] , v}_Y,
		\end{split}
		\]
		since $\dt , \ \A  , \ \B[\cdot, \us], \ \B[\us, \cdot]$ are (weakly) continuous from $X$ to $Y$, due to the boundedness of the operators shown at the beginning of the proof.
		
		It is therefore left to show that $u(0) = g$ holds, which we do by closely following the argument in \cite[Section 7.1.2 c]{Evans:2010}.
		First of all, since $X \embed C([0,T];W)$, we know that imposing an initial condition as $u(0)$ is sensible. Then by partial integration with respect to time in \eqref{eq:weakconvergence}, we have on the one hand for some $ v \in C^1([0,T];W) \subset L^2(0,T;V)$ with $v(T)=0$
		\[ 
		-\action{ \dt v, u}_Y +   \action{ \A u + \B[u,\us] + \B[\us, u] , v}_Y = \action{f , v }_Y + \action{ u(0), v(0) }_{V^*}
		\] 
		and on the other hand
		\[ 
		\lim_{n \to \infty} -\action{ \dt v, u_m}_Y +   \action{ \A u_m + \B[u_m,\us] + \B[\us, u_m] , v}_Y = \action{f , v }_Y + \action{g, v(0)}_{V^*}
		\] 
		since $u_m(0) \to g$ weakly in $V$. Comparing both expression we see that due to arbitrariness of $v(0)$, we obtain $u(0) = g$.
		\color{black}
		Uniqueness of $u$ follows by standard arguments, due to the linearity of \eqref{eq:full_system_abstract}, see \cite[Section 3.2.3]{QuadraticWave} for details.
	\end{proof}
	\subsection{Contraction argument for the bilinear problem} \label{se:nonlinearproblem}
	In this section, we establish well-posedness of the nonlinear problem \eqref{eq:full_system_abstract_nonlinear} using Theorem \ref{th:linearwellposed}.
	For the proof, we use the following version of the Newton-Kantorovich theorem to treat the quadratic nonlinearity $\B$ with a contraction argument.
	\begin{lemma}\label{th:NewtonKantorowich}(Newton-Kantorovich Theorem) \\
		Let $X$ and $Y$ be Banach spaces and $\F:\mathcal{D}(\F)\subseteq X \to Y$. Suppose that on an open convex set $D_*\subseteq \mathcal{D}(\F)$, $\F$ is Fr\'{e}chet differentiable and $\F'$ is Lipschitz continuous on $D_*$ with Lipschitz constant $K$
		\[
		\|\F'_{u_1}-\F'_{u_2}\|_{\linspa(X,Y)}\leq K \|u_1-u_2\|_{X} , \quad u_1, \, u_2\in D_*.
		\]
		For some $u_*\in D_*$, assume $\Gamma_*:={\F'^{-1}_{u_*}}$ is defined on all of $ \ Y$ and 
		\begin{equation}\label{betaK}
			\beta K \|\Gamma_*\F(u_*)\|_X\leq\frac12, \ \text{ where } \beta = \|\Gamma_*\|_{\linspa(Y,X)}.
		\end{equation}
		Set 
		\begin{equation}\label{rstart_rstarstar}
			r^\pm=\frac{1}{\beta K}(1\pm\sqrt{1-2\beta K\|\Gamma_*\F(u_*)\|_X})
		\end{equation}
		and suppose that $B_{r^-}^X(u_*)\subseteq D_*$. 
		Then, the Newton iterates 
		\begin{equation}\label{Newton}
			u^{(k+1)}=u^{(k)}-{\F'^{-1}_{u^{(k)}}}\F(u^{(k)}), \, k=0,1,2,\ldots 
		\end{equation}
		starting at $u^{(0)}:=u_*$ 
		are well-defined, lie in $B_{r^-}^X(u_*)$ and converge to a solution of $\F(u)=0$, which is unique in $B_{r^+}^X(u_*)\cap D_*$. 
	\end{lemma}
	A proof of Lemma \ref{th:NewtonKantorowich} can be found in e.g. \cite{Kantorovich1948,KantorovichAkilov}. \revision{With this, for any given $T>0$ we can establish well-posedness of the abstract system, provided the data are sufficiently small.}
	\begin{theorem} \label{th:nonlinear}
		\revision{Let $T>0$, $(f,g) \in L^2(0,T; D(\Asym^{\frac{1-\sigma}{2}} )^*) \times \dom(\Asym^\frac{\sigma}{2})$ and Assumption \ref{se:assumption} hold with $\sigma \in [0,1]$. Then, there exists a $r>0$ so that if} \Margin{3} $$ \| (\ff,g) \|_{L^2(0,T; D(\Asym^{\frac{1-\sigma}{2}} )^*) \times \dom(\Asym^\frac{\sigma}{2})} < r , $$  there exists a unique solution $u \in X := L^2(0,T;D(\Asym^{\frac{1+\sigma}{2}}) )  \cap H^1(0,T; D(\Asym^{\frac{1-\sigma}{2}} )^*)$ of equation \eqref{eq:full_system_abstract_nonlinear}.
	\end{theorem}
	\begin{proof}
		The proof follows a similar argumentation as the proof of \cite[Theorem 3.2]{QuadraticWave}. For convenience of the reader, we restate it here. As in Theorem \ref{th:linearwellposed}, we use the short notations 
		$Y:= L^2(0,T; D(\Asym^{\frac{1-\sigma}{2}} )^*)$,  $U:= D(\Asym^{\frac{1+\sigma}{2}})$, $V:=   D(\Asym^{\frac{1-\sigma}{2}} )$ and $W:=  \dom(\Asym^\frac{\sigma}{2}) $ 
		
		We seek to apply Lemma \ref{th:NewtonKantorowich} to the map $\F:X \to Y \times W$ defined as
		\[
		\F u := ( \dt u + \A u + \B[u,u], u(0)) - (f, g).
		\]
		We choose the open and convex set $D_* = U^{X}_{r_G}(0):= \{ u \in X : \| u \| < r_G \}$ with $r_G>0$ from Theorem \ref{th:linearwellposed} so that \eqref{eq:full_system_abstract} is well defined for all $\us \in D_*$.
		
		Then, by assumptions \ref{se:assc} \eqref{as:bilinestimate1} and \eqref{as:Linfemb} that yield boundedness of $\B:X\times X \to Y$
		we have
		\[ \label{eq:BFrechetDiff}
		\frac{\| \B[h, h] \|_Y }{ \| h \|_X} \leq \frac{ ( C_\B^2 C_T^2 \| h \|_X^2 \int_0^T \| h \|_U^2   ds )^\half }{ \| h \|_X} \leq {C_\B} C_T \| h \|_X \to 0 \quad ( \| h \|_X \to 0)
		\]
		meaning that $\F$ is Fr\'{e}chet differentiable at each $ \us \in D_*$ with
		\[
		\F'_\us : \xi \mapsto (\dt \xi + \A \xi, \xi(0) ) + ( \B[\xi, \us] + \B[u_*,\xi], 0),
		\]
		so that $ \F'^{-1}_\us : Y \times W \to X $ is well posed for $\us \in D_*$ by Theorem \ref{th:linearwellposed}. 
		
		By the same estimate for $\B$, we have for $u_1, u_2 \in D_*$ and $u \in X$
		\[
		\| (\F'_{u_1} - \F'_{u_2}) u \|_Y \leq \| \B[u, u_1 - u_2] + \B[u_1 - u_2, u] \|_Y \leq 2 C_\B C_T \| u_1 - u_2 \|_{X} \| u \|_{X} 
		\]
		giving Lipschitz continuity of $\F'$ with Lipschitz constant $K:=2 C_\B C_T$.
		
		Next, we choose $\us= 0 \in D_*$, in which case assumption \eqref{betaK} of Lemma \ref{th:NewtonKantorowich} turns out to be
		\[
		\beta K \| - (\dt \cdot + \A \cdot, \cdot(0) ) ^{-1}  (\ff, g) \|_{X} \leq \half  .
		\]
		The a priori estimate from Theorem \ref{th:linearwellposed} gives us $ \beta:= \| {\F'^{-1}_\us} \|_{\linspa{(Y \times W, X)}} \leq C_G $, so that $\| ( \dt \cdot + \A \cdot, \cdot(0))^{-1} \|_{\linspa(Y \times W , X)} \leq C_G$ by applying Theorem \ref{th:linearwellposed} with $\us=0 \in X$. Therefore, choosing 
		$
		r \leq \frac{1}{2 C_G^2 K} > 0
		$
		gives
		$$
		\beta K \| - ( \dt \cdot + \A \cdot, \cdot(0))^{-1} (\ff, g) \|_{X}
		\leq C_G^2 K \| (f,g) \|_{Y \times W} 
		< \half.
		$$
		Choosing additionally $r$ so small that $r^- < r_G$ and hence $B_{r^-}^X(0) \subseteq D_*$, the assumptions of Lemma \ref{th:NewtonKantorowich} hold, and we obtain in particular the stated result.
	\end{proof}
	\revision{We note that the above theorem does not provide a global well-posedness result, as the choice of $r$ depends on the embedding constant $C_T$, which can depend on $T$.} \Margin{3}
	\subsection{Nonhomogeneous boundary conditions} \label{nonhombcs}
In this section, we demonstrate how the results of the previous section can be adapted to establish well-posedness of \eqref{eq:full_system_abstract_nonlinear} in the presence of nonhomogeneous boundary conditions, using standard lifting arguments.

In the abstract formulation \eqref{eq:full_system_abstract_nonlinear} and under Assumption \ref{se:assumption}, no boundary conditions are explicitly imposed. However, when nonhomogeneous boundary conditions are incorporated into the operator formulation, assumption \ref{se:assb} may fail to hold. For instance, the symmetry of the linear operator 
$\Asym$ may be lost, as occurs with the Laplace operator subject to Dirichlet or Neumann boundary conditions.
	
	Therefore, the strategy in the following is to use extension properties of the self adjoint operator $\Asym$.
	Then, for a nonhomogeneous boundary value problem, we look for a solution in the bigger space $ \tild{X} \supset X$ such that $X \embed \tild{X}$ and assume that for a 
	suitable boundary condition operator $\bcop$ defined on $\tild{X}$, we have $\bcop(u)=0$ if $u \in X$. This means $X$ represents the space of functions with homogeneous boundary conditions. Equation \eqref{eq:full_system_abstract_nonlinear} can then be written as
	\[
	\begin{split} \label{eq:full_system_abstract_nonhom}
		\dt u + \tild{\A} u + \tild{\B}[u , u ] &=  \ff \\
		u(0) &= g \\
		\bcop(u) &= h,
	\end{split}
	\]
	where now $\tild{\A}: \tild{X} \to \tild{Y}$ and $\tild{\B}: \tild{X} \times \tild{X} \to \tild{Y}$.
	 To formalize this lifting argument, we make the following assumptions in the case of nonhomogeneous boundary conditions. 
	 \begin{assump}
	 	\label{eq:assumptions}
	 Let Assumption \ref{se:assumption} and Theorem \ref{th:linearwellposed} hold with $Y:= L^2(0,T; D(\Asym^{\frac{1-\sigma}{2}} )^*)$,  $U:= D(\Asym^{\frac{1+\sigma}{2}})$, $V:=   D(\Asym^{\frac{1-\sigma}{2}} )$ and $W:=  \dom(\Asym^\frac{\sigma}{2}) $. Additionally, assume that
		\begin{enumerate}[label=(\roman*),ref=\theassump(\roman*)]
		\item\label{as:a} there are Banach spaces $\tild{U}, \tild{V}, \tild{W}$ such that $U \embed \tild{U}$, $V \embed \tild{V}$ and $W \embed \tild{W}$ and $U,V,W$ are closed subspaces of $\tild{U}, \tild{V}, \tild{W}$, respectively. Define $$\tild{X} := L^2(0,T;\tild{U}) \cap H^1(0,T; {V}^*), \qquad \tild{Y}:= L^2(0,T;V^*).$$
		Moreover, the embedding $ \tild{X} \embed C([0,T];\tild{W})$ holds.
		\item\label{as:d} $\bcop : \tild{X} \to \bcop(\tild{X})$ is linear
		and $u \in \tild{X}$ with $\bcop(u)=0$ implies $u \in X$. 
		\item\label{as:b} $\tild{\A} = \tild{\Asym} + \tild{\Aant} : \tild{X} \to \tild{Y}$ linear is well defined, where $\tild{\Asym}: D(\tild{\Asym}) \subset H \to H$ with $D(\Asym) \embed D(\tild{\Asym})$ and $\tild{\Aant}: D(\tild{\Asym}) \subset H \to H$. Furthermore $\tild{\Asym}\big\vert_{X} = \Asym$ and $\tild{\Aant}\big\vert_{X} = \Aant$, while $\tild{\Aant}:\tild{X} \to \tild{Y}$ is bounded.
		
		$\tild{\B} : \tild{X} \times \tild{X} \to \tild{Y}$ bilinear is well defined, bounded 
	and $\tild{\B}\big\vert_{X \times X} = \B$.
		\item\label{as:c} there exists a unique $\tild{h} \in \tild{X}$ such that $\tild{\Asym} \tild{h} = 0$ and $\bcop(\tild{h}) = h$ 
		with $\| \tild{h} \|_{\tild{X}} \leq c_h \| h \|_{X_\bcop}$ for some Banach space $X_\bcop \subseteq \bcop(\tild{X})$.
	\end{enumerate}
\end{assump}
	We note that, in principle, it is possible to choose $X_\bcop = \bcop(\tild{X})$. However, for the application to nonlinear acoustics discussed in Section~\ref{se:application}, we assume more regularity for $h$ then $\bcop(\tild{X})$ to avoid technical details that would otherwise arise.
	\begin{theorem} \label{th:wellposednonhom}
		Let $T>0$ and Assumptions \ref{se:assumption}, \ref{eq:assumptions} hold.
		\revision{Let $ f \in Y$, $g \in \tild{W}$, $h \in X_\bcop$ with $ \tild{h}(0) - g \in W$. Then there exists a $\tild{r}>0$ such that if }
		\Margin{3}
		$$\| f \|_{{Y}} + \| g \|_{\tild{W}} + \| h \|_{X_\bcop} < \tild{r},
		$$ there exists a unique solution $u \in \tild{X} $ of equation \eqref{eq:full_system_abstract_nonhom}.
	\end{theorem}
	\begin{proof}
		The main idea  of the proof is to deduce a solution $u$ of \eqref{eq:full_system_abstract_nonhom} as $\you = \youn + \youE$ with $\youn \in X$ and $\youE \in \tild{X}$. 
		
		1. We wish to obtain $\youE$ as the unique solution of the linear problem
		\[ \label{eq:dirextension}
		\begin{split}
			\dt \youE + \tild{\Asym} \youE + \tild{\Aant} \youE  &= 0 \\
			\youE(0) &= g \\
			\bcop(\youE) &= h .
		\end{split}
		\]
		
		Using $\tild{h} \in \tild{X}$ from assumption \ref{as:c} 
		and plugging the ansatz $\youE = u_Z + \tild{h}$ into \eqref{eq:dirextension} gives an equation for $u_Z \in X$
		\[ \label{eq:dirichletharmonicextension}
		\begin{split}
			\dt u_Z + \A u_Z &=  - \dt \uD - \tild{\Aant} \uD \\
			u_Z(0) &= g - \tild{h}(0)  \\
			\bcop(u_Z) &= 0,
		\end{split}
		\]
		where we have $\tild{\A} u_Z = \A u_Z$, since $\bcop u_Z = 0$ using assumption \ref{as:d}.
		
		To prove the existence of a unique $u_Z \in X$ solving \eqref{eq:dirichletharmonicextension}, we apply Theorem \ref{th:linearwellposed} with $\us=0$, since by assumptions \ref{as:b}, \ref{as:c}
		$$	( - \dt \uD - \Aant \uD , \ g - \tild{h}(0) ) \in  {Y} \times {W}.$$
		This also yields the a priori estimate $\| u_Z \|_X \leq C_Z ( \| \dt \tild{h} - \tild{\Aant} \tild{h} \|_{Y} + \| g - \tild{h}(0) \|_W )$ for some $C_Z>0$. By construction $\tild{u}:= u_Z + \tild{h}$ indeed fulfills \eqref{eq:dirextension} and by the triangle inequality
		\[ \label{eq:utildestimate} 
		\| \tild{u} \|_{\tild{X}} \leq
		C_Z ( \| \dt \tild{h} \|_Y + \| \tild{\Aant} \|_Y + \| g \|_{\tild{W}} + \| \tild{h}(0) \|_{\tild{W}} ) + c_h \| h \|_{X_\bcop} \leq \tild{C} ( \| g \|_{\tild{W}} + \| h \|_{X_\bcop}), 
		\]
		where $\tild{C}>0$ depends on the operator norms of $\tild{\Aant}, \dt$ and on $c_h, C_T$.
		
		2. We proceed to construct $u^0 \in X$. Plugging the ansatz $\you = u^0 + \youE$ into \eqref{eq:full_system_abstract_nonhom} gives a system for $u^0$
		\newcommand{\prez}{p^0}
		\newcommand{\velz}{v^0}
		\[ 
		\begin{split}
			\dt u^0 + \tild{\A} u^0 + \tild{\B}[u^0 + \youE, u^0 + \youE] &= \ff \\
			u^0(0) &= 0 \\
			\bcop(u^0) &= 0,
		\end{split}
		\]
		which by bilinearity of $\B$ and since $\bcop(u^0)=0$ implies $\tild{\B}[u^0,u^0]=\B[u^0,u^0]$, can be written as
		\[ \label{eq:nonhomdirex}
		\begin{split}
			\dt u^0 + \A u^0 + {\B}[u^0, u^0] + \tild{\B}[u^0, \youE] + \tild{\B}[ \youE, u^0] &= \ff - \tild{\B}[\youE,\youE] \\
			u^0(0) &= 0 \\
			\bcop(u^0) &= 0.	\end{split}
		\]
		The strategy is now similar to the proof of Theorem \ref{th:nonlinear}. We show that for fixed $\tild{u} \in \tild{X}$ there exists a unique $u^0 \in X$ such that $\G u^0 = 0$ with $\G: X \to Y \times W$ defined as
		\[ \G u^0 := (\dt u^0	+ \A u^0 + \B[u^0, u^0] + \tild{\B}[u^0, \youE] + \tild{\B}[ \youE, u^0] , u^0(0)) - (  \ff - \tild{\B}[\youE,\youE],0).
		\]
		First of all, we note that $\tild{\B}[\youE, \youE] \in Y$, so that $\G$ is well defined. 
		Furthermore, we have that
		\[
		\G'_{\tild{\uss}} \cdot = (\dt \cdot + \A \cdot + \B[ \cdot, \tild{\uss}] + \B[\tild{\uss}, \cdot] + \tild{\B}[\cdot, \youE] + \tild{\B}[ \youE, \cdot], \cdot(0) )  
		\]
		for each $\tild{\uss} \in X$, since the only nonlinear term $\B[u^0,u^0]$ in $\G$ can be differentiated like the corresponding term in $\F$ in the proof of Theorem \ref{th:nonlinear}, see equation \eqref{eq:BFrechetDiff}. The Lipschitz continuity of $\G'$ with constant $K=2 C_\B C_T$ follows in the same way.
		
		To show that
		$\G'^{-1}_\uss: Y \times W \to X$ is well posed with $ \tild{ \uss} = 0 \in \tild{D_*}:=U^X_{\tild{r_G}}(0)$ for some $\tild{r_G}>0$ small enough, we can proceed analogously to the proof of Theorem \ref{th:linearwellposed}.

An exception is that in the energy estimate of $\B$, cf. \eqref{eq:energyestimateB}, we now have to treat the term  $\tild{\B}[u^0,\tild{u}]+ \tild{\B}[\tild{u},u^0]$. 
		We can replace the $ \| \us \|_{X}$ terms by $ \| \tild{u} \|_{\tild{X}}$ in \eqref{eq:energyestimateB} and the estimate can be carried out analogously in this case, since the assumptions \eqref{as:bilinestimate1}, \eqref{as:Linfemb} are assumed to hold for the tilde spaces, respectively. 
		
		In the end, we obtain a unique $u^0 \in X$ and constants $\tild{r}_G, \tild{C_G} > 0$ for which the a priori estimate 
		$ \| u^0 \|_X \leq \tild{C_G} \| f - \tild{\B}[\tild{u}, \tild{u}] \|_Y$
		is valid if $ \tild{u} \in B^{\tild{X}}_{\tild{r_G}}(0)$, which can be achieved due to \eqref{eq:utildestimate} by assuming that $\| g \|_{\tild{W}} + \| h \|_{X_\bcop}$ is small enough.
		
		In order to apply Lemma \ref{th:NewtonKantorowich} to $\G$, we note that, since $\beta \leq \tild{C_G}$
		\[
		\beta K \| - ( \dt \cdot + \A \cdot + \B[ \cdot,\tild{u}] + \B[\tild{u},\cdot], \cdot(0) )^{-1}  (\ff - \tild{\B}[\tild{u}, \tild{u}], 0) \|_{X} \leq \tild{C_G}^2 K \| \ff - \tild{\B}[\tild{u}, \tild{u}] \|_Y < \half 
		\]
		and $r^- < \tild{r_G}$ can be ensured if $\tild{r}$ is chosen small enough. This follows from smallness of $\| f \|_Y$ in combination with the boundedness of $\tild{\B}$ and the estimate \eqref{eq:utildestimate}.
		
		It is left to show that $u = u^0 + \tild{u}$ is unique. If $\hat{u}$ is another solution to \eqref{eq:full_system_abstract_nonhom}, then it can be decomposed into $\hat{u}=\hat{u}^0 + \hat{\tild{u}}$, where $\hat{\tild{u}}$ satisfies \eqref{eq:dirextension} and $\hat{u}^0:=\hat{u} - \tild{u}$ \eqref{eq:nonhomdirex}.    Since these equations have unique solutions, it holds $u = \hat{u}$.
		\color{black}
	\end{proof}
	\section{Application to nonlinear acoustics} \label{se:application}
	To apply the results of section \ref{se:abstractsystem} to equation \eqref{eq:full_system}, we note that the differential operators in \eqref{eq:full_system} can be split into a linear and bilinear part. Therefore, equation \eqref{eq:full_system}
	can be reformulated as
	\[ \label{eq:full_system_dir}
	\begin{split} 
		\dt u + \A u + \B[u ,u ] &= \ff \\
		u(0) &= g
	\end{split}
	\]
	with $u = (\pre, \vel)^\transpose, \ \ff = (\pref, \velf)^\transpose, \ g = (\pre_0, \vel_0)^\transpose$. The linear part of \eqref{eq:full_system} is represented in \eqref{eq:full_system_dir} by
	\[ \label{eq:def_A}
	\A u =   \Asym u + \Aant u =   - \begin{pmatrix}
		\pvis \lapl \pre \\ \vvis \lapl \vel
	\end{pmatrix} + 
	\begin{pmatrix}
		\divv \vel \\ \grad \pre
	\end{pmatrix}
	\]
	and the bilinear part of \eqref{eq:full_system} is defined in \eqref{eq:full_system_dir} with $z= (q, w)^\transpose$ using the product rule as
	\[ \label{eq:def_B}
	\B[u,z] = 
	\begin{pmatrix}
		\alpha \ba[u,z] + \beta \bb[u,z] \\ \gamma \bc[u,z] + \delta \bd[u,z]
	\end{pmatrix}
	:=
	\begin{pmatrix}
		\alpha q \divv v + \beta w \cdot \grad p \\
		\gamma q \grad p + \delta (\grad v) w
	\end{pmatrix},
	\]
	with $\alpha= \nonlinpar, \ \beta = \gamma = \delta = 1$, but in the following analysis we allow $\alpha, \beta, \gamma, \delta$ to be arbitrary nonzero real numbers.
	Note that in $\B$, first and zero order differential operators are applied to the first and second argument $u$ and $z$, respectively.
	\subsection{Preliminaries}
	We assume throughout this work that $\Omega$ is a bounded, simply connected domain and has at least $C^{2,1}-$boundary.
	In the following, we often write function spaces as $L^2, \ H^1, \ \dots$ without specifying the domain. If we do not specify the domain, then by default these spaces have the domain $\domain$. To note down vector valued function spaces, we use the notation $X^{j}:= X \times \dots \times X$, $j \in \N,$ with $X$ a Banach space and equip these spaces with the component-wise norm, i.e., $\| x \|_{X^j} = \sqrt{ \| x_1\|_{X}^2 + \dots + \| x_j \|^2_X }$. 
	An element $x \in X^j$ is often denoted as $x=(x_1, ..., x_j)^\transpose$.
	We use the following version of Hölder's inequality
	\begin{align}
		\| u v w \|_{ L^1 } \leq \| u \|_{ L^p } \| v \|_{ L^q} \| w \|_{L^r} \quad u \in L^p, \ v \in L^q, \ w \in L^r, \quad 1 = \frac{1}{p} + \frac{1}{q} + \frac{1}{r}
	\end{align}
	for $p,q,r\in[1,\infty)$, where $L^p$ denote the usual Lebesgue spaces. For integrals of $u \in L^1$ over $\Omega$, we use the shorthand notation $\int_\Omega u$, while for integrals over $\partial \Omega$, we write $\int_{\boundary \domain} u dS$, where $dS$ indicates the integration over a boundary.
	
	With $H^m$ we denote the usual Sobolev space of integrable functions with $m$-th order distributional derivatives in $L^2$. For $s>0$ we denote by $H^s$ the fractional order Sobolev spaces obtained by interpolation with the usual norm and seminorm. 
	For $u \in (H^1)^j$ we denote by $\jac u \in (L^2)^{j \times j}$ the (weak) Jacobian matrix. If $u \in H^1_0$, which is defined as the completion of $C^\infty_0$ in $H^1$, we choose $\| u \|_{H^1_0} = \| \grad u \|_{L^2} = \sqrt{\action{\grad u, \grad u}}$, where $\action{\cdot, \cdot}$ denotes the $L^2$ inner product. In this case, for smooth enough domains we have by Poincaré's inequality $ \| u \|_{L^2} \leq C_P \| \grad u \|_{(L^2)^d} $ with constant $C_P>0$ and space dimension $d \in \{2,3\}$.
	As $H^{-s}$ we denote the dual space of $H^s_0$. 
	
	Let $ \confun:= \{ u \in L^2 : u(x)=c, \ c \in \R \} $ be the closed subspace of constant functions of $L^2$. If $H$ is a Hilbert space and $M \subset H$ a subspace of $H$, we denote the quotient space as $H_{\diagup M}$.
	\subsubsection{Dirichlet Laplacian} 
	The Dirichlet Laplacian $- \lapl_D : \dom( - \lapl_D) \subset L^2 \to L^2$ can be obtained by a Friedrich's extension, see \cite[Theorem 10.19]{schmudgen2012}. If we assume $\Omega$ to be $C^2$-regular, we have $\dom(- \lapl_D) = H^2 \cap H^1_0 $ by elliptic regularity. Furthermore, $- \lapl_D$ is self adjoint and has a spectral decomposition as required in assumption \ref{se:assb} for $\Asym $ with $H=L^2$. 
	
	Moreover, the following characterization of the fractional domains of $- \lapl_D$ for $s \in [0,2]$ holds
	\[ \label{eq:dirchar}
	\dom( - \lapl_D^{\frac{s}{2}}) = \Hdir^s := \begin{cases}
		H^s \quad &  0 < s < \half,\\
		H^\half_{00} \quad &s = \half,\\
		H^s_0 \quad &\half < s < 1,\\
		H^s \cap H^1_0  \quad &1 \leq s \leq 2,
	\end{cases}
	\]
	see \cite[Theorem 3.1]{fefferman2022} in combination with \cite[Proposition 4.1]{bonito2013}. The space denoted by $H^\half_{00}$ consists of all $u \in H^\half$ such that $\int_\Omega d^{-1}  u^2 < \infty$, where $d(x) : \Omega \to \R$ stands for the Euclidean distance from $x$ to $\partial \Omega$.  We define $\Hdir^{-s}:=(\Hdir^s)^*$.
	
	Note that $\Hdir^s \embed H^s$ for $s \geq 0$ with embedding constant $C_{\Hdir}>0$, which follows from \cite[Proposition 2.2]{antil2017} and the fact that $H^1_0 \embed H^1$.
	\subsubsection{Neumann Laplacian}
	The classical Neumann Laplacian $- \lapl_N : \dom(\lapl_N) \subset L^2 \to L^2$ can be obtained by a Friedrich's extension, see \cite[Theorem 10.20]{schmudgen2012}.	If $\Omega$ is assumed to be $C^2$-regular, we have
	\[
	\dom(- \lapl_N):= \{ u \in H^2 :  \T( \grad u) \cdot n = 0  \},
	\]
	where $\T:H^s \to H^{s - \half}(\boundary \domain)$ if $s>\half$ denotes the classical trace operator, see \cite[Chapter 5.5]{Evans:2010}. 
	$-\lapl_N$ is also self adjoint and has almost a spectral decomposition as required in assumption \ref{se:assb} if $H=L^2$ with the exception that the first eigenvalue vanishes, i.e., $\lambda_1 = 0$. This can be overcome by considering $-\lapl_N:\dom( - \lapl_N)_{\diagup \confun} \subset L^2_{\diagup \confun} \to L^2 _{\diagup \confun}$, i.e., by factoring out constant functions. We write for this operator $-\lapl_{\tild{N}}$ with $\dom(-\lapl_{\tild{N}}) = \dom( - \lapl_N)_{\diagup \confun}$. Then, this Neumann Laplacian $-\lapl_{\tild{N}}$ does not have a vanishing eigenvalue, since the eigenspace corresponding to $\lambda_1$ for $-\lapl_N$ equals $C$.
	
	Therefore, fractional powers of $-\lapl_{\tild{N}}$ can be defined as in the Dirichlet case, and we have for $s \in [0,2]$ 
	\[ \label{eq:neuchar} 
	\dom(- \lapl_{\tild{N}}^\frac{s}{2}) = \Hneu^s 
	:= \begin{cases}
		\{ u \in H^s: \int_\Omega u  = 0 \} \quad &s < \frac{3}{2}, \\
		\{ u \in H^\frac{3}{2} : \int_\Omega  u = 0 , \  | \grad u | \in H^\half_{00} \} \quad &s = \frac{3}{2},  \\
		\{ u \in H^s: \int_\Omega u = 0,  \ \T( \grad u) \cdot n = 0 \} \quad &s > \frac{3}{2},
	\end{cases}
	\]
	see \cite[Proposition 2.4]{antil2017} for $s<1$ and \cite[Remark 2.10]{kim2020} for $s \geq 1$. Note that the condition $\int_\domain u = 0$ is equivalent to factoring out constant functions, and we define $\Hneu^{-s}:=(\Hneu^s)^*$. \\
	
	\subsubsection{Hodge/Lions Laplacian}
	For Hodge/Lions boundary conditions, we only cover  the case $d = 3$ for simplicity, so that the vector cross product $\times$ is a map  $ \times : \R^d \times \R^d \to \R^d$. We consider a vector valued Laplacian acting on functions with $d=3$ components fulfilling the Hodge/Lions boundary conditions as in \cite{mitrea2004, monniaux2013}, denoted by $- \lapl_M$ in the following. This operator can be constructed via a self adjoint Friedrich's extension with domain
	\[ \dom( - \lapl_M) = \{ v \in H^1_\perp : \divv v \in H^1, \ \curl (\curl v) \in (L^2)^d, \ (\curl v ) \times n = 0 \text{ on } \boundary \Omega \} \]
	with $$
	\dom( - \lapl_M^\half)=H^1_\perp:=\{ v \in (L^2)^d : \divv v \in L^2, \ \curl v \in (L^2)^{d}, \ v \cdot n  = 0 \text{ on } \partial \Omega\}
	$$ in case of a Lipschitz-regular domain $\Omega$, see \cite[Section 4]{monniaux2013}. When $\Omega$ is assumed to be more regular, one can obtain more regular function spaces. This is shown in the next Lemma, justifying our notation $H^1_\perp$. Note that this construction of $\lapl_M$ and the representation of $D(-\lapl_M)$ is linked to the vector calculus formula $\lapl v = \grad \divv v - \curl \curl v$. 
	\begin{lemma} \label{le:divcurl}
		Let $\Omega$ be $C^{1,1}$-regular and $d=3$. Then for $a \in H^\half(\partial \domain) $ and $ b \in H^\half(\boundary \domain)^{d}$ the continuous embeddings
		\[ \label{eq:embedH1}
		\begin{split}
			\{  &v \in (L^2)^d : \divv v \in L^2, \ \curl v \in (L^2)^{d}, \ v \cdot n = a \text{ on }  \boundary \domain \} \embed (H^1)^d,
			\\
			\{  &v \in (L^2)^d : \divv v \in L^2, \ \curl v \in (L^2)^{d}, \ v \times n = b \text{ on }  \boundary \domain \} \embed (H^1)^d.
		\end{split}
		\]
		hold.
		If $\Omega$ is $C^{2,1}$-regular, then for $c \in H^\frac{3}{2}(\boundary \domain)$ the continuous embedding
		\[ \label{eq:embedH2}  \{ v \in (H^1)^d : \divv v \in H^1 , \ \curl v \in (H^1)^{d}, \ v \cdot n = c \text{ on }  \boundary \domain  \} \embed (H^2)^d \]
		holds,
		leading to $\dom( - \lapl_M) \embed (H^2)^d$.
	\end{lemma}
	\begin{proof}
		The embeddings \eqref{eq:embedH1}, \eqref{eq:embedH2} can be deduced from \cite{amrouche1998} using \cite[Theorems 2.9, 2.12, Remark 2.14 and Corollary 2.15 with $m=2$]{amrouche1998}. 
		
		To prove $\dom( - \lapl_M) \embed (H^2)^d$ note that for $w := \curl v$ with $v \in \dom( - \lapl_M)$ we have $ \divv w = \divv \curl v = 0 \in L^2$ (in a distributional sense), $\curl w = \curl \curl v \in (L^2)^d$ and $w \times n = \curl v \times n = 0$ on $\boundary \domain$. Therefore, applying  \eqref{eq:embedH1} gives us $w \in (H^1)^d$, that is $\curl  v \in (H^1)^d$. Thus, using \eqref{eq:embedH2} we obtain $\dom( - \lapl_M) \embed (H^2)^d$.
	\end{proof}
	To the best knowledge of the author, no concrete characterization of $\dom( - \lapl_M^s)$ for $s \notin \N$ in terms of fractional Sobolev spaces is proven in the literature. This is ultimately the reason, why we consider nonhomogeneous Hodge/Lions boundary conditions only for the case $\sigma=1$, cf. Theorems \ref{th:wellposedneu} and \ref{th:dir-hod}.
	
	Finally, the following Lemma is a well-posedness result for $- \lapl_M$ following from \cite[Theorem 1.2]{mitrea2004}. 
	\begin{lemma} \label{le:hodgeregularity}
		Let $\Omega$ be $C^1$-regular and simply connected with $d=3$. Then, for $- \half < s < \half$ and any $\chi \in (H^s)^d, \ \xi \in H^{s - \half}(\partial \domain)$ and $$\nu \in H^s_\times:= \{ v \in  H^{s - \half}(\partial \Omega)^d : v = u \times n \text{ on } \boundary \domain \text{ for some } u \in (H^s)^d \text{ with } \curl u \in (H^s)^d \} $$ there exists a unique $v$ with 
		$$ \curl v \in (H^s)^{d}, \quad \curl \curl v \in (H^s)^d, \quad \divv v \in H^{s+1}$$ 
		such that
		\[ \label{eq:hodgeproblem}
		\begin{split} 
			- \lapl v &= \chi \\
			v \cdot n &= \xi \text{ on }\boundary \Omega\\
			(\curl v) \times n &= \nu \text{ on } \boundary \Omega.
		\end{split}
		\] 
		Additionally, it holds for some $C>0$ \[ \label{eq:hodgewellposed} 
		\begin{split}
		\| \curl v \|_{(H^s)^d} + \| \curl \curl v \|_{(H^s)^d} &+ \| \divv v \|_{H^{s+1}} \\ &\leq C ( \| \chi \|_{(H^s)^d} + \| \xi \|_{H^{s- \half}(\boundary \domain)} + \| \nu \|_{H^{s - \half}(\boundary \domain)^d} ). 
				\end{split}
		\]
	\end{lemma}
	\begin{proof}
		The statement is a special case of \cite[Theorem 1.2]{mitrea2004}. Within the notation of \cite{mitrea2004}, we set $p=2$ and choose $\mathcal{M} = \R^3$ with Euclidean metric as manifold. Then $\Omega$ is a Lipschitz domain of $\R^3$ and in the case of $p=2$ and $C^1$-regular $\domain$ the valid choices of $s$ in \cite[Theorem 1.2]{mitrea2004} according to \cite[(1.1)]{mitrea2004} are characterized by
		\[ - \half < s < \half, \quad 0 < \half - \frac{s}{3} < 1 . \]
		The estimate \eqref{eq:hodgewellposed} follows from the fact that according to \cite[Theorem 1.2]{mitrea2004} the boundary value problem \eqref{eq:hodgeproblem} is Fredholm solvable of index zero, which in particular implies that the solution operator $(\chi, \xi, \nu ) \mapsto v$ corresponding to \eqref{eq:hodgeproblem} is bounded.
		\color{black}
	\end{proof}
	\subsection{Dirichlet-Dirichlet boundary conditions}	
	In order to obtain well-posedness of \eqref{eq:full_system} under the boundary condition \eqref{bc-a}, we consider
	\[ \label{eq:full_system_dir1}
	\begin{split} 
		\dt u + \A u + \B[u ,u ] &= \ff \\
		u(0) &= g \\
		\bcop_D(u) &= h ,
	\end{split}
	\]
	where $\bcop_D: (H^s)^{d+1} \to H^{s - \half}(\boundary \domain)^{d+1}$ for $s>\half$ denotes the operator containing the classical trace operator $\T$ in each component, i.e., $\bcop_D(p,v)= (\T(p), \T(v_1), \dots, \T(v_d))^\transpose$.
	\begin{theorem} \label{th:wellposeddir}
		Let $T>0$, $\sigma \in [\half, 1]$, and 
		\[ f \in L^2(0,T; H^{{\sigma-1}})^{d+1}, \ \ g \in (H^{{\sigma}})^{d+1}, \ h \in H^1(0,T;H^{\sigma + \half}(\boundary \Omega) )^{d+1} 
		\] small enough with $(\tild{h}(0) - g) \in (\Hdir^{{\sigma}})^{d+1}$, where $\tild{h}$ is the Dirichlet harmonic extension of $h$. Then, there exists a unique solution 
		$$u \in L^2(0,T;H^{{\sigma+1}})^{d+1} \cap H^1(0,T;\Hdir^{{\sigma-1}})^{d+1} $$ of equation \eqref{eq:full_system_dir1}.
	\end{theorem}
	\begin{proof}
		We first check Assumption \ref{se:assumption} choosing $H:=(L^2)^{d+1}$ as underlying Hilbert space with scalar product $\action{ u, v} = \int_{\Omega} u \cdot v.$
		\begin{enumerate} 	[leftmargin=*]
		\item Assumption \ref{se:assa}: Follows since $L^2$ is separable and infinite dimensional. 
		\item 
		Assumption \ref{se:assb}:
		We define \[ \Asym : \dom( - \lapl_D)^{d+1} \subset H \to H, \ u=(p,v_1, \dots, v_d)^\transpose \mapsto  (
		- \zeta \lapl_D p , - \mu \lapl_D v_1 , \dots ,  - \mu \lapl_D v_d )^\transpose
		, \] 
		which is self adjoint, since $ - \lapl_D$ on $\dom(- \lapl_D)$ is. Since $H^1_0 \embed H^1$ compactly for $C^2$-regular domains and $(L^2)^{d+1}$ is infinite dimensional, self adjointness is equivalent to the desired spectral decomposition, see \cite[Proposition 5.12]{schmudgen2012}. For the smallest eigenvalue $\lambda_1$ of $\Asym$ we have $\lambda_1> \min\{\zeta, \mu\} C_P^{-1}$.
		Since $\dom( - \lapl_D^\frac{s}{2}) = (L^2, H^2 \cap H^1_0)_s$ by Remark \ref{re:fractional}, we obtain $\dom(\Asym^\frac{s}{2}) = (\Hdir^s)^{d+1}$ using the characterization \eqref{eq:dirchar} of $D(-\lapl_D^s)$.
		
		For $\Aant : (H^1_0)^{d+1} \subset (L^2)^{d+1} \to (L^2)^{d + 1}, \ u=(p, v)^\transpose \mapsto (\divv v, \grad p)^\transpose$ we have by partial integration, cf. \cite[Theorem D.8]{schmudgen2012},
		\[ 
		\action{\Aant u , \tild{u}} = \int_\Omega \divv v \tild{p} + \grad p \cdot \tild{v} = - \int_\Omega \grad \tild{p} \cdot v + \divv \tild{v} p + \int_{\partial \Omega} (v \tild{p} + p \tild{v}) \cdot n dS  = \action{u, - \Aant \tild{u}}
		\]
		for $u , \tild{u} \in (H^1_0)^{d+1} = \dom(\Asym^\half) \supseteq \dom(\Asym)$. Choosing $u=\tild{u}$ and using symmetry of $\action{\cdot, \cdot}$ yields the required relation \eqref{as:anti}.
		
		Boundedness of $\Aant : (H^1_0)^{d+1} \to (L^2)^{d + 1}$ with $C_{\Aant} \leq 1$ follows, since $\| \grad p \|_{(L^2)^d} = \| p \|_{H^1_0}$ and $\| \divv v \|_{(L^2)} \leq \| v \|_{(H^1_0)^{d}}$ yields $ \| \Aant u \|_{(L^2)^{d+1}} \leq \| u \|_{(H^1_0)^{d+1}}$.
		
		\item Assumption \ref{se:assc}: We define $\B$ as in \eqref{eq:def_B}. To obtain the estimate \eqref{as:bilinestimate1} note that $  H^s \embed L^{\frac{6}{3-2s}}$ with constant $C_{\embed}(s, \domain)>0$ for $d \leq 3$ and $0 \leq s < \frac{3}{2}$, see e.g. \cite{dinezza2012} and the references therein. So using H\"older's inequality with $p_1 = p_2 = \frac{6}{3-2 \sigma}, p_3 = \frac{6}{3-2 (1-\sigma)} $ gives
		\[ \label{eq:estimateBdir}
		\begin{split}
			\left \vert \int_\Omega \B[u,w] \cdot v \right \vert  
			&\leq		C(\alpha,\beta,\gamma,\delta) \| \grad u \|_{(L^{p_1})^{d+1}} \| w \|_{(L^{p_1})^{d+1}} \| v\|_{(L^{p_3})^{d+1}} \\
			&\leq
			C(\alpha,\beta,\gamma,\delta) C_{\embed}^3 \| \grad u \|_{(H^{{\sigma}})^{d+1}} \| w \|_{(H^\sigma)^{d+1}} \| v\|_{(H^{1-\sigma})^{d+1}} 
			\\
			&\leq
			C(\alpha,\beta,\gamma,\delta) C_{\embed}^3 \|u \|_{H^{({\sigma+1}})^d} \| w \|_{(H^\sigma)^{d+1}} \| v \|_{(H^{1-\sigma})^{d+1}} 
			\\
			&\leq C(\alpha,\beta,\gamma,\delta) C_{\embed}^3 C_{\Hdir}^3 \| u \|_{(\Hdir^{{\sigma+1}})^{d+1}} \| w \|_{(\Hdir^{\sigma})^{d+1}} \| v \|_{(\Hdir^{{1-\sigma}})^{d+1}}
		\end{split}
		\]
		with $C(\alpha,\beta, \gamma, \delta)>0$,
		since
		$ {3 -  2 \sigma} + {3 - 2 \sigma } + {3 - 2 ( 1 - \sigma} ) \leq 6 $
		is fulfilled for $\sigma \geq \half$. 
		
		\end{enumerate}
		Hence, we can apply Theorem \ref{th:nonlinear} under the assumptions that $f,g$ are sufficiently small with $g \in \Hdir^\sigma$ to obtain well-posedness of \eqref{eq:full_system_dir1} if $h=0$ with the spaces $U=(\Hdir^{\sigma+1})^{d+1}$, $W=(\Hdir^\sigma)^{d+1}$ and $V=(\Hdir^{1-\sigma})^{d+1}$, yielding
		${X} = L^2(0,T;\Hdir^{\sigma+1})^{d+1} \cap H^1(0,T; \Hdir^{\sigma-1})^{d+1}$ and $ Y = L^2(0,T;\Hdir^{\sigma-1})^{d+1}.$ 
		\\
		
		To treat the nonhomogeneous case $h \neq 0$, we additionally check assumption \ref{eq:assumptions} in order to apply Theorem \ref{th:wellposednonhom}. 
		\begin{enumerate}	[leftmargin=*]
	\item Assumption \ref{as:a}:
		We choose $\tild{U}=(H^{1+\sigma})^{d+1}$, $\tild{V}=(H^{1-\sigma})^{d+1}$, $\tild{W}=(H^\sigma)^{d+1}$ and  
		the required embeddings follow naturally, since $\Hdir^s \embed H^s$ for $s \geq 0$. 
		This gives us $\tild{X} = L^2(0,T;H^{{\sigma+1}})^{d+1} \cap H^1(0,T;\Hdir^{{\sigma-1}})^{d+1}$ and $\tild{Y}=Y$.
		
		Combining two Gelfand embeddings with $\cup$ we obtain
		\begin{align}
		(L^2(0,T;\Hdir^{1+s}) \cap H^1(0,T;\Hdir^{s - 1}) ) \cup (	L^2(0,T;H^{1+s}) 
		\cap H^1(0,T;(H^{1-s})^*) ) \\ \embed C(0,T;\Hdir^s) \cup  C(0,T;H^s) = C(0,T;H^s)
	\end{align}
which implies, since $(H^{1-s})^* \embed \Hdir^{s-1}$ for $s \leq 1$,
	\begin{align}
	(L^2(0,T;\Hdir^{1+s}) \cap H^1(0,T;\Hdir^{s - 1}) ) \cup (	L^2(0,T;H^{1+s}) 
	\cap H^1(0,T;\Hdir^{s - 1}) ) \embed C(0,T;H^s).
\end{align}
This leads to
\begin{align}
	(L^2(0,T;\Hdir^{1+s}) \cup   	L^2(0,T;H^{1+s}) 
	 ) \cap H^1(0,T;\Hdir^{s - 1}) \embed C(0,T;H^s), 
\end{align}
		which gives the modified embedding $\tild{X} \embed C(0,T;\tild{W})$.
		
		\item Assumption \ref{as:d}: $\bcop_D$ is linear and 
		$u \in \tild{X}$ with $\bcop_D(u) = 0$ implies $u \in X$. This follows, since $(H^1_0)^{d+1} \subset X$ and the fact that functions $u \in H^1$ that fulfill $\mathcal{T}(u)=0$ are in $H^1_0$, c.f \cite[Chapter 5.5, Theorem 2]{Evans:2010}. 
			
\item	Assumption \ref{as:b}:	
		For $\tild{\Asym}$ we choose $\dom(\tild{\Asym})=(H^{1+\sigma})^{d+1} \subset H$ and define $\tild{\Asym},\tild{\Aant}$ with the same formulas as $\Asym,\Aant$ earlier in the proof, so that $\tild{\Asym}\big\vert_{ X} = \Asym$ and $\tild{\Aant}\big\vert_{ X} = \Aant$ follow naturally.
		
		For the well definedness of $\tild{\Asym}$, we note that $-\lapl_D$ can be seen as an operator mapping functions from $H^{s+1} \to (H^{1-s})^* \subseteq \Hdir^{s - 1}$ for $s\leq 1$. 
		
		Boundedness of $\tild{ \Aant }:\tild{X} \to \tild{Y}$ holds, since for $u \in \tild{X}$ we have
		$$ \| \tild{\Aant} \tild{u} \|_{{Y}} \leq C_* \| \tild{\Aant} \tild{u} \|_{L^2(0,T;L^2)^{d+1}} \leq C_* C_\Aant \| \tild{u} \|_{L^2(0,T;H^1_0)^{d+1}} \leq C_* C_\Aant \| \tild{u} \|_{\tild{X}},$$ where $C_*$ stands for the embedding constant of $ L^2(0,T;L^2)^{d+1} \embed {Y}$.
		
		We define $\tild{\B}$ via the same formula as $\B$ earlier in the proof, so that $\tild{\B}\big\vert_{ X \times X}=\B$ holds naturally.
		
		Note that estimate \eqref{as:bilinestimate1} also holds for the tilde norms as one can see in \eqref{eq:estimateBdir}. Therefore, by the same argument as in the proof of Theorem \ref{th:linearwellposed} we conclude boundedness of $\tild{\B}:\tild{X} \times \tild{X} \to Y$.  
\item	Assumption \ref{as:c}: To construct the demanded extension $\tild{h}$ we use the bounded harmonic Dirichlet extension operator $D^{\lapl}: H^{s - \half}(\boundary \Omega) \to H^{s}$ for $s > \half$ as used in \cite{kaltenbacher2011b}. Defining with $h=(h_1,...,h_{d+1})^\transpose \in H^1(0,T;H^{\sigma + \half}(\boundary \Omega) )^{d+1} $ 
		$$\tild{h} := ( D^{\lapl} h_1, ... , D^{\lapl} h_{d+1} )^\transpose \in H^1(0,T;H^{\sigma+1})^{d+1}$$ we have by construction $\tild{A} \tild{h} = \lapl \tild{h} = 0$ with $\bcop_D(\tild{h})=h$. Due to boundedness of $D^\lapl$ and the embedding $H^1(0,T;H^{\sigma+1}) \embed \tild{X}$ it holds for some $C,c_h>0$
		$$ \| \tild{h} \|_{\tild{X}} \leq C \| \tild{h} \|_{H^1(0,T;H^{\sigma+1})^{d+1}} \leq c_h \| h \|_{H^1(0,T;H^{\sigma+\half}(\boundary \domain) )^{d+1}}. $$
		The compatibility condition $\tild{h}(0) - g \in {W}$ is fulfilled by the assumptions specifying the valid choices for $g,h$ in the above theorem.
	\end{enumerate}
Hence, choosing $f,g,h$ small enough we can apply Theorem \ref{th:wellposednonhom}.
	\end{proof}	
	Note that for $\sigma=1$ we obtain the results from \cite{QuadraticWave} if $h=0$. For $\sigma>\half$ the condition $\tild{h}(0) - g \in (\Hdir^{{\sigma}})^{d+1}$ translates into $\bcop_D(h(0))=\bcop_D(g)$, which is a natural compatibility assumption. In the border case $\sigma=\half$ we obtain well-posedness with least regular data as possible. In this case there does not exist a trace of the initial condition, but in order to compensate for this, the condition $\tild{h}(0) - g \in H^\half_{00}$ is required. 
	\subsection{Neumann-Hodge/Lion boundary conditions}
	In order to obtain well-posedness of \eqref{eq:full_system} under boundary condition \eqref{bc-b}, we consider
	\[ \label{eq:full_system_neu1}
	\begin{split} 
		\dt u + \A u + \B[u ,u ] &= \ff \\
		u(0) &= g \\
		\bcop_N(u) &= h, 
	\end{split}
	\]
	where $\bcop_N:(H^s)^{d+1} \to H^{s - \frac{3}{2}}(\boundary \domain) \times H^{s - \frac{1}{2}}(\boundary \domain) \times H^{s - \frac{3}{2}}(\boundary \domain)^{d} $ incorporates the Neumann-Hodge/Lions boundary conditions, i.e.,
	$$\bcop_N(u) = \bcop_N(p,v) = ( \T(\grad p) \cdot n , \ \T(v) \cdot n, \ \T(\curl v) \times n)^\transpose .$$ 
	\begin{theorem} \label{th:wellposedneu}
		Let $T>0$, $d=3$, and 
		$$f \in L^2(0,T; (L^2_{\diagup \confun})^* \times (L^2)^d), \ g \in (H^1)^{d+1}, \ h \in H^1(0,T;H^{\half}(\boundary \domain) \times H^\frac{3}{2}(\boundary \domain) \times H_\times^\half )
		$$ 
		small enough with $ ( \tild{h}(0) - g ) \in  H^1 \times H^1_\perp$, where $\tild{h}$ is the harmonic extension of $h$ with respect to $\bcop_N$. Then, there is a unique solution 
		$$u \in  L^2(0,T;H^{2}_{\diagup \confun} \times (H^2)^d) \cap H^1(0,T; (L^2_{\diagup \confun})^* \times (L^2)^d) $$ of \eqref{eq:full_system_neu1}. 
	\end{theorem}
	\begin{remark}
		We remark that factoring out constant functions from the solution space is reasonable in the case of nonlinear acoustics, since equation \eqref{eq:full_system} only describes how fluctuations of physical variables evolve over time.  
	\end{remark}
	\begin{proof}
		We first check assumption \ref{se:assumption} with $H = L^2_{\diagup \confun} \times (L^2)^d$ as underlying Hilbert space with scalar product $\action{ u, v} = \int_{\Omega} u \cdot v$. Let $\sigma \in (\half,1]$.
		\begin{enumerate}[leftmargin=*]
		\item Assumption \ref{se:assa}: Holds, since $L^2_{\diagup \confun}$ is still infinite dimensional and a Hilbert space.

		
		\item Assumption \ref{se:assb}: We define
		\[ \Asym : \dom(\Asym) \subset H \to H, \ u=(p,v) \mapsto  \begin{pmatrix}
			- \zeta \lapl_{\tild{N}} p \\ - \mu \lapl_M v 
		\end{pmatrix} \] 
		with $
		\dom(\Asym) = \dom( - \lapl_N)_{\diagup \confun} \times \dom( - \lapl_M) 
		$. Since $- \lapl_N$ and $- \lapl_M$ are self adjoint, we obtain that $\Asym$ is also self adjoint. Since $H^1_{\diagup \confun} \embed L^2$ and $H^1_\perp \embed  (L^2)^d$ compactly, see \cite[Theorem 2.8]{amrouche1998}, we obtain that $\Asym$ has pure point spectrum, see \cite[Proposition 5.12]{schmudgen2012}. Since we exclude constant solutions from $- \lapl_N u = \lambda u$, we only obtain strictly positive eigenvalues of $- \lapl_N$. 
		For $- \lapl_M$ we remark that assuming $\Omega$ is simply connected, yields that there is no eigenvalue $\lambda_1$ with $\lambda_1=0$, since a Poincaré like inequality holds in $\dom( - \lapl_M^\half)$, cf \cite[Corollary 3.16]{amrouche1998}. The result that the kernel of $- \lapl_M$ is trivial if $\Omega$ is simply connected also follows from \cite[Theorem 1.2]{mitrea2004} for less regular domains. 
		
		For $\Aant : \dom(\Asym) \subset H \to H, \ u=(p, v)^\transpose \mapsto (\divv v, \grad p)^\transpose$ by partial integration as in the proof of Theorem \ref{th:wellposeddir} we have
		$\action{\Aant u , \tild{u}} 
		= \action{u, - \Aant \tild{u}},$
		for $u , \tild{u} \in  \dom(\Asym^\half) = H^1_{\diagup \confun} \times H^1_\perp $, since $v \cdot n$ vanishes on $\partial \domain$ for $ v \in H^1_\perp$. The required boundedness of $\Aant$ follows again as in the proof of Theorem \ref{th:wellposeddir}, since $H^1_\perp \embed (H^1)^d$ can be equipped with the $H^1$-norm.
		
		\item Assumption \ref{se:assc}: To show \eqref{as:bilinestimate1} we use the definition of $\B$ and estimate \eqref{eq:estimateBdir} as obtained in the proof of Theorem \ref{th:wellposeddir}. Then, it remains to show that the $(H^r)^{d+1}$-norms can be estimated by the $D(\Asym^\frac{r}{2})$ norms. Indeed, we have $\dom(\Asym^\frac{s}{2}) \embed (H^s)^{d+1}$, 
		since 
		\[ \label{eq:neumannembed}
		\dom(\Asym^{s}) = ( (L^2)^{d+1}, \dom(\Asym))_{{s}} \embed (L^2, H^2)_s^{d+1} = (H^{2s})^{d+1}
		\] 
		using Remark \ref{re:fractional}. Note that $\dom(\Asym) \embed (H^2)^{d+1}$ follows from $\dom(- \lapl_N) \embed H^2$ and Lemma \ref{le:divcurl} relying on the assumed $C^{2,1}$-regularity of $\Omega$.  
	\end{enumerate}
	Hence, we can apply Theorem \ref{th:nonlinear} under the assumptions that $f,g$ are sufficiently small with $f \in L^2(0,T;\dom(\Asym^\frac{1-\sigma}{2})^*), \ g \in \dom(\Asym^{\frac{\sigma}{2}})$ to obtain well-posedness of \eqref{eq:full_system_neu1} if $h=0$. 
	
	For $\sigma=1$ we have the explicit representation of the spaces $U=\Hneu^2 \times \dom(-\lapl_M)$, 
	$W= \Hneu^1 \times D(-\lapl_M)^\half = H^1 \times H^1_\perp$ 
	and $V= L^2_{\diagup \confun} \times (L^2)^d$, yielding
${X} = L^2(0,T; \Hneu^2 \times \dom(-\lapl_M) ) \cap H^1(0,T; (L^2_{\diagup \confun})^* \times (L^2)^d)$ and $ Y = L^2(0,T;(L^2_{\diagup \confun})^* \times (L^2)^d).$ \\

In order to treat nonhomogeneous boundary conditions in the case of $\sigma=1$, we check assumption \ref{eq:assumptions} to apply Theorem \ref{th:wellposednonhom}. 
\begin{enumerate} [leftmargin=*]
		\item Assumption \ref{as:a}: We choose $\tild{U}=H^{2}_{\diagup \confun} \times (H^{2})^{d}$, $\tild{V}= (L^2)^{d+1}$, $\tild{W}=(H^1)^{d+1}$ and  
		the required embeddings hold, see embedding \eqref{eq:neumannembed} and the fact that $L^2_{\diagup \confun} \embed L^2$. This gives us $\tild{X} = L^2(0,T;H^{2}_{\diagup \confun} \times (H^2)^d) \cap H^1(0,T; (L^2_{\diagup \confun})^* \times (L^2)^d) $ and $\tild{Y}=Y$ from above.
		
		The modified embedding \eqref{as:Linfemb} can be shown as follows. 
		From a Gelfand embedding we have that
		\[ \label{eq:embedd}
		L^2(0,T;(H^2)^d) \cap H^1(0,T;(L^2)^d) \embed C(0,T;(H^1)^d) \] 
		and from the homogeneous case we obtain in the first component for $\sigma=1$ the embedding
		\[ \label{eq:embeddd} L^2(0,T; \Hneu^2) \cap H^1(0,T;(L^2_{\diagup \confun})^*) \embed C(0,T;\Hneu^1). \]
		Combining embedding \eqref{eq:embedd} with $d=1$ and \eqref{eq:embeddd} using $\cup$, yields  
		\[ \label{eq:embedddd} L^2(0,T;H^2_{\diagup \confun}) \cap H^1(0,T;(L^2_{\diagup \confun})^*) \embed C(0,T;H^1). \] 
		Therefore, since the embeddings \eqref{eq:embedd} and \eqref{eq:embedddd} hold, the modified embedding holds as well.
		\color{black}
		\item Assumption \ref{as:d}: $\bcop_N$ is linear and 
		$u \in \tild{X}$ with $\bcop_N(u) = 0$ implies $u \in X$ by construction of $\bcop_N$.
		\item Assumption \ref{as:b}: For $\tild{\Asym}$ we choose $\dom(\tild{\Asym})= H^2_{\diagup \confun} \times (H^{2})^d \subset H$ and define $\tild{\Asym},\tild{\Aant}$ with the same formulas as $\Asym,\Aant$ earlier in the proof, so that $\tild{\Asym}\big\vert_{ X} = \Asym$ and $\tild{\Aant}\big\vert_{X} = \Aant$ follow naturally.
		
		For the well definedness of $\tild{\Asym}$, we note that $-\lapl_N$ can be seen as an operator mapping functions from $H^2_{\diagup \confun} \to L^2 \subset (L^2 _{\diagup \confun})^*$ and $-\lapl_M$ mapping from $(H^2)^d \to (L^2)^d$. 
		
		Boundedness of $\tild{ \Aant }:\tild{X} \to \tild{Y}$ holds, since due to $L^2(0,T;L^2 )^{d+1} \embed \tild{Y}$ with embedding constant $C_*$, we have as in the proof of Theorem \ref{th:wellposeddir} for $u \in \tild{X}$
		$$ \| \tild{\Aant} \tild{u} \|_{{\tild{Y}}} \leq C_* \| \tild{\Aant} \tild{u}  \|_{{{L^2(0,T;L^2)^{d+1}}}} \leq C_* C_\Aant \| \tild{u} \|_{L^2(0,T;H^1_0)^{d+1}} \leq C_* C_\Aant \| \tild{u} \|_{\tild{X}}. $$ 
		
		We define $\tild{\B}$ via the same formula as $\B$ earlier in the proof, so that $\tild{\B}\big\vert_{ X \times X}=\B$ naturally.
		
		Note that estimate \eqref{as:bilinestimate1} also holds for the tilde norms as one can see from \eqref{eq:estimateBdir} and the fact that $H^2$ and $H^2_{\diagup \confun}$ share the same norm. Therefore, by the same argument as in the proof of Theorem \ref{th:linearwellposed} we conclude boundedness of $\tild{\B}:\tild{X} \times \tild{X} \to Y$.  
		\item Assumption \ref{as:c}: For the boundary data $h$ we compactly write $h=(h_1,h_2,h_3)$ in the following. To construct the demanded extension $\tild{h}=(\tild{h}_p,\tild{h}_v)^\transpose$ we use for $\tild{h}_p$ the Neumann harmonic extension operator $N^\lapl: H^{s- \frac{3}{2}}(\boundary \domain) \to H^s$, cf. \cite{kaltenbacher2011a}, yielding as extension of $h_1 \in H^1(0,T;H^\half(\boundary \domain))$ a unique $\tild{h}_p := N^\lapl h_1 \in H^1(0,T;H^{2}_{\diagup \confun})$.
		To obtain $\tild{h}_v$, we use Lemma \ref{le:hodgeregularity} with $s=0$, $\chi=0$, $\xi = h_2$, $\nu = h_3$ giving a unique $\tild{h}_v$ such that $$\curl \tild{h}_v \in (L^2)^d, \ \curl \curl  \tild{h}_v \in (L^2)^d, \ \divv  \tild{h}_v \in H^{1}$$
		for each time $t \in [0,T]$.
		Using \eqref{eq:embedH1} from Lemma \ref{le:divcurl}, we actually obtain that $\tild{h_v} \in (H^1)^d$ for each time and by using the arguments in the proof of Lemma \ref{le:divcurl} we have that $\tild{h}_v \in H^1(0,T;H^{2})^d$, due to the regularity of $h_2,h_3$ and $\domain$. Altogether, due to the boundedness of the extension operators involved, we obtain for some $c_h>0$ the required bound $$ \| \tild{h} \|_{\tild{X}} \leq c_h \| h \|_{H^1(0,T;H^{\half}(\boundary \domain) \times H^\frac{3}{2}(\boundary \domain) \times H_\times^\half )}.$$
		\color{black}
	\end{enumerate}
Hence, choosing $f,g,h$ small enough we can apply Theorem \ref{th:wellposednonhom}.
	\end{proof}
	\subsection{Neumann-Dirichlet boundary conditions}
	To treat boundary condition \eqref{bc-c}, we consider the system
		\[ \label{eq:full_system_noslip}
	\begin{split} 
		\dt u + \A u + \B[u ,u ] &= \ff \\
		u(0) &= g \\
		\bcop_S(u) &= h
	\end{split}
	\]
	with $$\bcop_S(p,v) = ( \T(\grad p) \cdot n , \ \T(v_1), \ \dots, \  \T(v_d) )^\transpose. $$ 
	\begin{theorem} \label{th:wellposednoslip}
	Let $T>0$, $\sigma \in (\half, 1]$, and 
	\[ f \in L^2(0,T; H^{{\sigma-1}})^{d+1}, \ \ g \in (H^{{\sigma}})^{d+1}, \ h \in H^1(0,T;H^{\sigma - \half}(\boundary \Omega) 
	\times H^{\sigma + \half}(\partial \Omega)^d )
	\] small enough with $(\tild{h}(0) - g) \in \Hneu^\sigma \times (\Hdir^{{\sigma}})^{d}$, where $\tild{h}= ( \tild{h_1}, \tild{h_2}, ..., \tild{h}_{d+1})^\transpose$ and $\tild{h}_1$ is the Neumann harmonic extension of $h_1$ and $\tild{h}_2$ to $\tild{h}_{d+1}$ are Dirichlet harmonic extensions of $h_1$ to $h_{d+1}$, respectively. Then, there exists a unique solution 
	$$u \in  L^2(0,T;H^{{1 + \sigma }}_{\diagup \confun} \times (H^{1+ \sigma})^d) \cap H^1(0,T; \Hneu^{\sigma - 1} \times (\Hdir^{\sigma-1} )^d) $$ of equation \eqref{eq:full_system_noslip}.
\end{theorem}
\begin{proof}
		The proof is effectively a combination of the arguments in the proofs of Theorems \ref{th:wellposeddir} and \ref{th:wellposedneu}, but for the sake of completeness, we elaborate the steps here in detail. 
		
		We first check assumption \ref{se:assumption} with $H = L^2_{\diagup \confun} \times (L^2)^d$ as underlying Hilbert space with scalar product $\action{ u, v} = \int_{\Omega} u \cdot v$. 
	\begin{enumerate}[leftmargin=*]
		\item Assumption \ref{se:assa}: Holds, as in the proof of Theorem \ref{th:wellposedneu}.
		\item Assumption \ref{se:assb}: We define
		\[ \Asym : \dom(-\lapl_{\tild{N}}) \times \dom( - \lapl_D)^{d} \subset H \to H, \ u=(p,v_1, \dots, v_d)^\transpose \mapsto  (
		- \zeta \lapl_{\tild{N}} p , - \mu \lapl_D v )^\transpose
		, \] 
		with $
		\dom(\Asym) = \dom( - \lapl_N)_{\diagup \confun} \times \dom( - \lapl_D)^d 
		$. Since $- \lapl_{\tild{N}}$ and $- \lapl_D$ are self adjoint, we obtain that $\Asym$ is also self adjoint. Since $H^1_0 \embed H^1$ and $H^1_{\diagup \confun} \embed L^2$ compactly for $C^2$-regular domains self adjointness is equivalent to the desired spectral decomposition, see arguments in the proofs of Theorems \ref{th:wellposeddir}, \ref{th:wellposedneu}.
		
		For $\Aant : \dom(\Asym) \subset H \to H, u=(p, v)^\transpose \mapsto (\divv v, \grad p)^\transpose$ by partial integration as in the proof of Theorem \ref{th:wellposeddir} we have
		$\action{\Aant u , \tild{u}} 
		= \action{u, - \Aant \tild{u}}$
		for $u , \tild{u} \in  \dom(\Asym^\half) = H^1_{\diagup \confun} \times H^1_0 $, since $v$ vanishes on $\partial \domain$ for $ v \in H^1_0$. The required boundedness of $\Aant$ follows again as in the proof of Theorem \ref{th:wellposeddir}. 
		\item Assumption \ref{se:assc}: Follows from estimate \eqref{eq:estimateBdir}, since $\Hneu^s \embed H^s$ and $\Hdir^s \embed H^s$. 
		\color{black}
	\end{enumerate}
	Hence, we can apply Theorem \ref{th:nonlinear} under the assumptions that $f,g$ are sufficiently small with $f \in L^2(0,T; \Hneu^{\sigma - 1} \times (\Hdir^{\sigma-1})^d), \ g \in \Hneu^\sigma \times (\Hdir^\sigma)^d$ to obtain well-posedness of \eqref{eq:full_system_noslip} if $h=0$. By \eqref{eq:dirchar} and \eqref{eq:neuchar} we obtain the spaces $U= \Hneu^{\sigma+1} \times (\Hdir^{\sigma+1})^{d}$, $W=\Hneu^\sigma \times (\Hdir^\sigma)^{d}$ and $V= \Hneu^{1-\sigma} \times (\Hdir^{1-\sigma})^d$, yielding
	${X} = L^2(0,T;\Hneu^{\sigma+1} \times (\Hdir^{\sigma+1})^d) \cap H^1(0,T; \Hneu^{\sigma-1} \times (\Hdir^{\sigma-1})^d)$ and $ Y = L^2(0,T;\Hneu^{\sigma-1} \times (\Hdir^{\sigma-1})^d).$  
	
	In order to treat nonhomogeneous boundary conditions, we check assumption \ref{eq:assumptions} to apply Theorem \ref{th:wellposednonhom}. 
	\begin{enumerate} [leftmargin=*]
		\item Assumption \ref{as:a}: We choose $\tild{U}=H^{1+\sigma}_{\diagup \confun} \times (H^{1+\sigma})^{d}$, $\tild{V}=(H^{1-\sigma})^{d+1}$, $\tild{W}=(H^\sigma)^{d+1}$ and  
		the required embeddings follow. 
		This gives us $\tild{X} = L^2(0,T; H^{\sigma+1}_{\diagup \confun } \times H^{{\sigma+1}})^{d} \cap H^1(0,T;\Hneu^{\sigma-1} \times \Hdir^{{\sigma-1}})^{d}$ and $\tild{Y}=Y$ from above.
		The modified embedding \eqref{as:Linfemb} can again be shown by comining the modified embeddings of the previous proofs of Theorems \ref{th:wellposeddir}, \ref{th:wellposedneu}.  
		\item Assumption \ref{as:d}: $\bcop_S$ is linear and 
		$u \in \tild{X}$ with $\bcop_S(u) = 0$ implies $u \in X$ by construction of $\bcop_S$.
		\item Assumption \ref{as:b}: For $\tild{\Asym}$ we choose $\dom(\tild{\Asym})= H^2_{\diagup \confun} \times (H^{2})^d \subset H$ and define $\tild{\Asym},\tild{\Aant}$ with the same formulas as $\Asym,\Aant$ earlier in the proof, so that $\tild{\Asym}\big\vert_{ X} = \Asym$ and $\tild{\Aant}\big\vert_{X} = \Aant$ follow naturally.
		
		For the well definedness of $\tild{\Asym}$, we note that also $-\lapl_N$ can be seen as an operator mapping functions from $H^{s+1}_{\diagup \confun} \to (H^{1-s})^* \subseteq \Hneu^{s - 1}$ for $s\leq 1$. 
		
		Boundedness of $\tild{ \Aant }:\tild{X} \to \tild{Y}$ holds, by arguments of the previous proofs, as well as the assumptions on $\tild{\B}$.
		\item Assumption \ref{as:c}:To construct the demanded extension $\tild{h}$ we use the Dirichlet and Neumann harmonic extension operators as in the proofs of Theorems \ref{th:wellposeddir} and \ref{th:wellposedneu}.
		We set $$\tild{h} := ( N^{\lapl} h_1, D^{\lapl} h_2 , \dots, D^{\lapl} h_{d+1} )^\transpose \in H^1(0,T;H^{\sigma+1})^{d+1}$$ and have by construction $\lapl \tild{h} = 0$ with $\bcop_S(\tild{h})=h$. Due to boundedness of $D^\lapl, N^\lapl$ and the embedding $H^1(0,T;H^{\sigma+1}) \embed \tild{X}$ it holds for some $C,c_h>0$
		$$ \| \tild{h} \|_{\tild{X}} \leq C \| \tild{h} \|_{H^1(0,T;H^{\sigma+1})^{d+1}} \leq c_h \| h \|_{H^1(0,T;H^{\sigma-\half}(\boundary \domain) \times H^{\sigma + \half}(\boundary \domain)^d )}. $$  
	\end{enumerate}
	Hence, choosing $f,g,h$ small enough we can apply Theorem \ref{th:wellposednonhom}.
\end{proof}	
\subsection{Dirichlet-Hodge/Lions boundary conditions}
	To treat boundary condition \eqref{bc-d}, we consider the system
\[ \label{eq:full_system_bc-d}
\begin{split} 
	\dt u + \A u + \B[u ,u ] &= \ff \\
	u(0) &= g \\
	\bcop_R(u) &= h
\end{split}
\]
with $$\bcop_R(p,v) = ( \T( p) , \ \T(v) \cdot n, \ \T(\curl v) \times n)^\transpose. $$ 
\begin{theorem} \label{th:dir-hod}
	Let $T>0$, $d=3$, and 
$$f \in L^2(0,T; L^2)^{d+1}, \ g \in (H^1)^{d+1}, \ h \in H^1(0,T;H^{\frac{3}{2}}(\boundary \domain) \times H^\frac{3}{2}(\boundary \domain) \times H_\times^\half )
$$ 
small enough with $ ( \tild{h}(0) - g ) \in  H^1_0 \times H^1_\perp$, where $\tild{h}$ is the harmonic extension of $h$ with respect to $\bcop_R$. Then, there is a unique solution 
$$u \in  L^2(0,T;H^2)^{d+1} \cap H^1(0,T; L^2)^{d+1} $$ of \eqref{eq:full_system_bc-d}. 
\end{theorem}
\begin{proof}
	The proof is omitted, as it consists of a straightforward combination of the previous proofs.
\end{proof}	
	\section{Conclusion}
	To conclude, we have established well-posedness of \eqref{eq:full_system} for the boundary condition configurations Dirichlet–Dirichlet, Dirichlet–Hodge/Lions, Neumann–Hodge/Lions, and Neumann–Dirichlet.
	For cases involving Dirichlet or Neumann boundary conditions, we have obtained well-posedness under lower regularity assumptions on the data in dimensions two and three.
	When Hodge/Lions boundary conditions are present, we restrict the results to dimension three and standard regularity.
	Once the abstract framework from Section~\ref{se:abstractsystem} is in place, well-posedness follows directly upon choosing appropriate domains for the operators. 
	
	\revision{All of the well-posedness results are established under the assumption that the data are sufficiently small. If the solution possesses enough regularity, the embedding in \ref{as:Linfemb} is not required, and global well-posedness can be inferred, as shown in \cite{QuadraticWave} for homogeneous Dirichlet boundary conditions. It is expected that these results carry over to the case of inhomogeneous boundary conditions as well.}
	\Margin{3}

	
There are several natural directions for future research.
A first step would be to compute numerical solutions of \eqref{eq:full_system} and compare the behavior under the different boundary condition configurations.
Another question is how the present results extend to less regular domains, such as Lipschitz domains, which form the natural setting for finite element methods.
Less smooth boundaries imply less smooth domains of the Laplace operator, making it relevant in what sense the fractional domains of the Laplacian can be characterized by Sobolev spaces in analogy to e.g. $\Hdir^s$.
Another open task in this direction would be to characterize the fractional domains of $-\lapl_M$ for smooth and Lipschitz boundaries. 

Other types of boundary conditions, such as absorbing boundary conditions, see \cite{Shevchenko2015}, are also of interest in nonlinear acoustics.
It remains an open question whether analogous boundary conditions can be formulated for equation~\eqref{eq:full_system}.

\revision{Concerning the limit of vanishing damping coefficients 
$\zeta$ and $\mu$ in \eqref{eq:full_system}, note that, as seen in the energy estimate \eqref{eq:energyestimateB}, the nonlinear terms must be controlled using the regularity provided by the symmetric operator $\Asym$. Consequently, for low regularity data, the controllability of the nonlinear terms relies on the Laplace operator. In contrast, if the initial data are sufficiently regular, well-posedness for the case $\zeta = \mu = 0$ can be established using hyperbolic methods. This scenario, among other aspects, is studied in the unfinished work \cite{LehnerMeliani2026}.} 
\Margin{5}
	\section*{Acknowledgment}
The author gratefully acknowledges Barbara Kaltenbacher for her careful reading of the manuscript and for her valuable suggestions for its improvement
	This work is supported by the Austrian Science Fund (FWF) [10.55776/P36318].
\bibliographystyle{plain}
\bibliography{lit}

@book{Ladyzhenskaya1968,
  author     = {Lady{\v{z}}enskaja, O. A. and Solonnikov, V. A. and Ural'ceva, N. N.},
  title     = {Linear and Quasilinear Equations of Parabolic Type},
  series    = {Translations of Mathematical Monographs, Vol.\ 23},
  publisher = {American Mathematical Society},
  address   = {Providence, R.I.},
  year      = {1968},
  pages     = {xi+648},
  isbn      = {978‑0‑8218‑1573‑1},
  doi       = {10.1090/mmono/023},
}

@article{Ribaud1998,
  title = {Cauchy problem for semilinear parabolic equations with initial data in {$H^s_p(\mathbb{R}^n)$}},
  volume = {14},
  ISSN = {2235-0616},
  url = {http://dx.doi.org/10.4171/RMI/232},
  DOI = {10.4171/rmi/232},
  number = {1},
  journal = {Revista Matemática Iberoamericana},
  publisher = {European Mathematical Society - EMS - Publishing House GmbH},
  author = {Ribaud,  F.},
  year = {1998},
  pages = {1–46}
}

@book{lunardi1995,
  author       = {Lunardi, A.},
  title        = {Analytic Semigroups and Optimal Regularity in Parabolic Problems},
  series       = {Progress in Nonlinear Differential Equations and Their Applications},
  volume       = {16},
  publisher    = {Birkhäuser, Basel},
  year         = {1995},
  pages        = {xvii+424},
  isbn         = {3-7643-5172-1},
}

@book{pazy1983,
  author       = {Pazy, A.},
  title        = {Semigroups of Linear Operators and Applications to Partial Differential Equations},
  series       = {Applied Mathematical Sciences},
  volume       = {44},
  publisher    = {Springer-Verlag, New York},
  year         = {1983},
  pages        = {x + 279},
  isbn         = {0-387-90845-5},
  doi          = {10.1007/978-1-4612-5561-1},
}

@article{Shevchenko2015,
  author       = {I. Shevchenko and B. Kaltenbacher},
  title        = {Absorbing boundary conditions for nonlinear acoustics: The {Westervelt} equation},
  journal      = {Journal of Computational Physics},
  volume       = {302},
  pages        = {200--221},
  year         = {2015},
  doi          = {10.1016/j.jcp.2015.08.051},
  isbn         = {1090-2716},
  publisher    = {Elsevier},
}

@article{QuadraticWave,
  title = {A First Order in Time Wave Equation Modeling Nonlinear Acoustics},
  author = {Kaltenbacher, B. and Lehner, P.},
  year = {2025},
  journal = {Journal of Mathematical Analysis and Applications},
  volume = {543},
  number = {2, Part 2},
  pages = {128933},
  issn = {0022-247X},
  doi = {10.1016/j.jmaa.2024.128933},
  urldate = {2024-11-14}
}

@article{mitrea2004,
  title = {Sharp {{Hodge}} Decompositions, {{Maxwell}}'s Equations, and Vector {{Poisson}} Problems on Nonsmooth, Three-Dimensional {{Riemannian}} Manifolds},
  author = {Mitrea, M.},
  year = {2004},
  journal = {Duke Mathematical Journal},
  volume = {125},
  number = {3},
  issn = {0012-7094},
  doi = {10.1215/S0012-7094-04-12322-1},
  urldate = {2024-09-10}
}

@article{dinezza2012,
  title = {Hitchhiker's Guide to the Fractional {{Sobolev}} Spaces},
  author = {Di Nezza, E. and Palatucci, G. and Valdinoci, E.},
  year = {2012},
  journal = {Bulletin des Sciences Math{\'e}matiques},
  volume = {136},
  number = {5},
  pages = {521--573},
  doi = {10.1016/j.bulsci.2011.12.004}
}

@incollection{kaltenbacher2011b,
  title = {Well-Posedness and {{Exponential Decay}} for the {{Westervelt Equation}} with {{Inhomogeneous Dirichlet Boundary Data}}},
  booktitle = {Parabolic {{Problems}}: {{The Herbert Amann Festschrift}}},
  author = {Kaltenbacher, B. and Lasiecka, I. and Veljovi{\'c}, S.},
  year = {2011},
  pages = {357--387},
  publisher = {Springer},
  address = {Basel},
  doi = {10.1007/978-3-0348-0075-4_19}
}

@article{amrouche1998,
  title = {Vector Potentials in Three-Dimensional Non-Smooth Domains},
  author = {Amrouche, C. and Bernardi, C. and Dauge, M. and Girault, V.},
  year = {1998},
  journal = {Mathematical Methods in the Applied Sciences},
  volume = {21},
  number = {9},
  pages = {823--864},
  issn = {0170-4214, 1099-1476},
  doi = {10.1002/(SICI)1099-1476(199806)21:9<823::AID-MMA976>3.0.CO;2-B},
  urldate = {2024-09-13},
  copyright = {http://doi.wiley.com/10.1002/tdm\_license\_1.1},
  langid = {english}
}

@article{bonito2013,
   ISSN = {00255718, 10886842},
 URL = {http://www.jstor.org/stable/24489199},
 author = {A. Bonito and J. E. Pasciak},
 journal = {Mathematics of Computation},
 number = {295},
 pages = {2083--2110},
 publisher = {American Mathematical Society},
 title = {NUMERICAL APPROXIMATION OF FRACTIONAL POWERS OF ELLIPTIC OPERATORS},
 urldate = {2024-09-17},
 volume = {84},
 year = {2015}
}

@book{lunardi2018,
  title = {{Interpolation Theory}},
  author = {Lunardi, A.},
  year = {2018},
  publisher = {Scuola Normale Superiore},
  address = {Pisa},
  doi = {10.1007/978-88-7642-638-4},
  urldate = {2024-09-17},
  copyright = {http://www.springer.com/tdm},
  isbn = {978-88-7642-639-1 978-88-7642-638-4},
  langid = {italian}
}

@book{schmudgen2012,
  title = {Unbounded {{Self-adjoint Operators}} on {{Hilbert Space}}},
  author = {Schm{\"u}dgen, K.},
  year = {2012},
  series = {Graduate {{Texts}} in {{Mathematics}}},
  volume = {265},
  publisher = {Springer Netherlands},
  address = {Dordrecht},
  doi = {10.1007/978-94-007-4753-1},
  urldate = {2024-09-04},
  copyright = {https://www.springernature.com/gp/researchers/text-and-data-mining},
  isbn = {978-94-007-4752-4 978-94-007-4753-1},
  langid = {english}
}

@misc{li2022b,
  title = {Characterization of Smooth Solutions to the {{Navier-Stokes}} Equations in a Pipe with Two Types of Slip Boundary Conditions},
  author = {Li, Z. and Pan, X. and Yang, J.},
  year = {2022},
  number = {arXiv:2110.02445},
  eprint = {2110.02445},
  primaryclass = {math},
  publisher = {arXiv},
  urldate = {2024-09-02},
  archiveprefix = {arXiv},
  langid = {english},
note = {arXiv:2110.02445v3 [math.AP]}
}

@article{antil2017,
  title = {Fractional Operators with Inhomogeneous Boundary Conditions: Analysis, Control, and Discretization},
  shorttitle = {Fractional Operators with Inhomogeneous Boundary Conditions},
  author = {Antil, H. and Pfefferer, J. and Rogovs, S.},
  year = {2018},
  journal = {Communications in Mathematical Sciences},
  volume = {16},
  number = {5},
  pages = {1395--1426},
  publisher = {International Press of Boston},
  issn = {1945-0796},
  doi = {10.4310/CMS.2018.v16.n5.a11},
  urldate = {2024-09-17},
  langid = {english}
}

@article{fefferman2022,
  title = {Simultaneous Approximation in {{Lebesgue}} and {{Sobolev}} Norms via Eigenspaces},
  author = {Fefferman, C. L. and Hajduk, K. W. and Robinson, J. C.},
  year = {2022},
  journal = {Proceedings of the London Mathematical Society},
  volume = {125},
  number = {4},
  pages = {759--777},
  issn = {0024-6115, 1460-244X},
  doi = {10.1112/plms.12469},
  urldate = {2024-08-26},
  langid = {english}
}

@article{kim2020,
  title = {{{Fractional Order Sobolev Spaces For The Neumann Laplacian And The Vector Laplacian}}},
  author = {Kim, S.},
  year = {2020},
  journal = {Journal of the Korean Mathematical Society},
  volume = {57},
  number = {3},
  pages = {721--745},
  doi = {10.4134/JKMS.J190351},
  urldate = {2024-08-27},
  langid = {english}
}

@article{kaltenbacher2011a,
  title = {Well-Posedness of the {{Westervelt}} and the {{Kuznetsov}} Equation with Nonhomogeneous {{Neumann}} Boundary Conditions},
  author = {Kaltenbacher, B. and Lasiecka, I.},
  year = {2011},
  journal = {Discrete and Continuous Dynamical
Systems Supplement},
  number = {2},
  pages = {763--773},
  publisher = {Conference Publications},
  issn = {0133-0189},
  doi = {10.3934/proc.2011.2011.763},
  langid = {english}
}

@article{xiao2007,
  title = {On the Vanishing Viscosity Limit for the {{3D Navier}}-{{Stokes}} Equations with a Slip Boundary Condition},
  author = {Xiao, Y. and Xin, Z.},
  year = {2007},
  journal = {Communications on Pure and Applied Mathematics},
  volume = {60},
  number = {7},
  pages = {1027--1055},
  issn = {0010-3640, 1097-0312},
  doi = {10.1002/cpa.20187},
  urldate = {2024-08-28},
  langid = {english}
}

@article{phan2017,
  title = {Gevrey Regularity for {{Navier}}--{{Stokes}} Equations under {{Lions}} Boundary Conditions},
  author = {Phan, D. and Rodrigues, S. S.},
  year = {2017},

  journal = {Journal of Functional Analysis},
  volume = {272},
  number = {7},
  pages = {2865--2898},
  issn = {0022-1236},
  doi = {10.1016/j.jfa.2017.01.014},
  urldate = {2024-09-17}
}

@article{monniaux2013,
  title = {Various Boundary Conditions for {{Navier-Stokes}} Equations in Bounded {{Lipschitz}} Domains},
  author = {Monniaux, S.},
  year = {2013},
  journal = {Discrete \& Continuous Dynamical Systems - S},
  volume = {6},
  number = {5},
  pages = {1355--1369},
  issn = {1937-1179},
  doi = {10.3934/dcdss.2013.6.1355},
  urldate = {2024-08-28},
  langid = {english}
}

@book{zeidler1995,
  title = {Applied {{Functional Analysis}}},
  author = {Zeidler, E.},
  editor = {Marsden, J. E. and Sirovich, L.},
  year = {1995},
  series = {Applied {{Mathematical Sciences}}},
  volume = {108},
  publisher = {Springer New York},
  address = {New York, NY},
  doi = {10.1007/978-1-4612-0815-0},
  urldate = {2024-08-28},
  copyright = {http://www.springer.com/tdm},
  isbn = {978-1-4612-6910-6 978-1-4612-0815-0},
  langid = {english}
}

@book{Hamilton:1997,
	address = {San Diego, Calif.},
	edition = {First},
	title = {Nonlinear {Acoustics}},
	isbn = {978-0-12-321860-5},
	language = {Englisch},
	publisher = {Academic Press},
	author = {Hamilton, M.F. and Blackstock, D.T.},
	year = {1997},
}

@book{Roubicek:2013,
	series = {International {Series} of {Numerical} {Mathematics}},
	title = {Nonlinear {Partial} {Differential} {Equations} with {Applications}},
	volume = {153},
	isbn = {978-3-0348-0512-4 978-3-0348-0513-1},
	url = {http://link.springer.com/10.1007/978-3-0348-0513-1},
	language = {en},
	urldate = {2024-02-12},
	publisher = {Springer Basel, Basel},
	author = {Roubíček, T.},
	year = {2013},
	doi = {10.1007/978-3-0348-0513-1}
}

@article{Kaltenbacher:2015,
	title = {Mathematics of nonlinear acoustics},
	volume = {4},
	issn = {2163-2480},
	url = {http://aimsciences.org//article/doi/10.3934/eect.2015.4.447},
	doi = {10.3934/eect.2015.4.447},
	language = {en},
	number = {4},
	journal = {Evolution Equations \& Control Theory},
	author = {Kaltenbacher, B.},
	year = {2015},
	pages = {447--491}
}

@book{Evans:2010,
	address = {Providence (R.I.)},
	edition = {second},
	series= {Graduate Studies in Mathematics},
	volume= {19},
	title = {Partial {Differential} {Equations}},
	isbn = {978-0-8218-4974-3},
	language = {Englisch},
	publisher = {American Mathematical Society},
	author = {Evans, L.C.},
	year = {2010},
}

@article{MeyerWilke2013,
  author    = {S. Meyer and M. Wilke},
  title     = {Global well-posedness and exponential stability for Kuznetsov’s equation in Lp-spaces},
  journal   = {Evolution Equations and Control Theory},
  year      = {2013},
  volume    = {2},
  pages     = {365--378},
}

@article{LehnerMeliani2026,
    author = {P. Lehner and M. Meliani},
    title = {Well-posedness of a first order in time formulation for nonlinear acoustics with fractional damping},
    year = {2026},
    note = {work in progress}
}

@incollection{DoerflerFindeisenWienersZiegler:2019,
url = {https://doi.org/10.1515/9783110548488-002},
title = {Parallel adaptive discontinuous Galerkin discretizations in space and time for linear elastic and acoustic waves},
booktitle = {Space-Time Methods: Applications to Partial Differential Equations},
author = {W. D\"orfler and S. Findeisen and C. Wieners and D. Ziegler},
pages = {61--88},
doi = {doi:10.1515/9783110548488-002},
isbn = {9783110548488},
year = {2019},
}

@article {Kantorovich1948,
    AUTHOR = {Kantorovich, L.V.},
     TITLE = {Functional analysis and applied mathematics},
   JOURNAL = {Vestnik Leningrad. Univ.},
  FJOURNAL = {Vestnik Leningrad. Univ.},
    VOLUME = {3},
      YEAR = {1948},
    NUMBER = {6},
     PAGES = {3--18},
   MRCLASS = {46.3X},
  MRNUMBER = {30699},
}

@book {KantorovichAkilov,
    AUTHOR = {Kantorovich, L.V. and Akilov, G.P.},
     TITLE = {Functional analysis},
   EDITION = {Second},
 PUBLISHER = {Pergamon Press, Oxford-Elmsford, N.Y.},
      YEAR = {1982},
     PAGES = {xiv+589},
      ISBN = {0-08-023036-9},
   MRCLASS = {46-01 (47-01 65-01)},
  MRNUMBER = {664597},
}

@article{bansal:2021,
	title = {A space-time {Trefftz} discontinuous {Galerkin} method for the acoustic wave equation in first-order formulation},
	volume = {138},
	issn = {0029-599X, 0945-3245},
	doi = {10.1007/s00211-017-0910-x},
	number = {2},
	urldate = {2024-03-12},
	journal = {Numer. Math.},
	author = {Moiola, A. and Perugia, I.},
	year = {2018},
	pages = {389--435},
}

@article{kuznetsov:1971,
	title={Equations of nonlinear acoustics},
	author={V. P. Kuznetsov},
	year={1971},
	number={16},
	pages={467--470}, 
	journal = {Soviet Physics: Acoustics},
}
\end{document}